\documentclass{amsart}
\usepackage{amsthm}
\usepackage{amsfonts}
\usepackage{mathrsfs}
\usepackage{amssymb} 
\usepackage{amsmath}
\usepackage{dsfont}
\usepackage{enumitem}
\usepackage{indentfirst}
\usepackage{graphicx}
\graphicspath{ {./} }
\usepackage{caption}
\usepackage{subcaption}
\usepackage{wrapfig}
\usepackage{setspace}
\usepackage{mathtools}
\usepackage[dvipsnames]{xcolor}
\usepackage{tikz}
\usepackage{tikz-cd}
\usetikzlibrary{braids}
\usetikzlibrary{shapes,arrows}
\usepackage{algorithm}
\usepackage{algpseudocode}
\usepackage{hyperref}

\author{Daniel Hothem}
\title{Extending Higher Bruhat Orders to non-longest words in $S_n$}

\newtheorem{theorem}{Theorem}
\numberwithin{theorem}{subsection}
\newtheorem{lemma}[theorem]{Lemma}

\theoremstyle{definition}
\newtheorem{example}[theorem]{Example}
\newtheorem{corollary}[theorem]{Corollary}
\newtheorem{definition}[theorem]{Definition}

\theoremstyle{remark}
\newtheorem{rmk}[theorem]{Remark}
\newtheorem{notation}[theorem]{Notation}

\let\emph\relax
\DeclareTextFontCommand{\emph}{\bfseries}
\DeclareMathOperator{\lex}{lex}
\DeclareMathOperator{\anti}{anti}
\DeclareMathOperator{\MS}{MS}
\DeclareMathOperator{\SSO}{SS}

\newcommand{\Inv}{\text{Inv}}

\newcommand{\rlex}{\rho_{\lex}}
\newcommand{\ranti}{\rho_{\anti}}
\newcommand{\sdiff}[1]{\setminus\lbrace{#1}\rbrace}
\newcommand{\flipup}[3]{#1 \nearrow_{#3} #2}
\newcommand{\paths}[3]{A_{#3}{(#1 \rightarrow #2)}}
\newcommand{\epaths}[3]{B_{#3}{(#1 \rightarrow #2)}}

\newcommand{\HB}[2]{B_{#1}(#2)}
\newcommand{\extOrd}[2]{\mathcal{O}^{#1}_{#2}}

\begin{document}
\begin{abstract}
	In this paper, we extend Manin and Schechtman's higher Bruhat orders for $S_n$ to higher Bruhat orders for non-longest words $w$ in $S_n$. We prove that the higher Bruhat orders of non-longest words are ranked posets with unique minimal and maximal elements. As in Manin and Schechtman's original paper, the $k$-th Bruhat order for $w$ is created out of equivalence classes of maximal chains in its $(k-1)$-st Bruhat order. We also define the second and third Bruhat orders for arbitrary realizable $k$-sets, and prove that the second Bruhat order has a unique minimal and maximal element. Lastly, we also outline how this extension may guide future research into developing higher Bruhat orders for affine type A Weyl groups.
\end{abstract}
\maketitle
\tableofcontents	

\section{Introduction}

\subsection{Preliminaries}

The symmetric group on $n$ strands $S_n$ is ubiquitous in mathematics. Here we view it as a Coxeter group in the standard way. The weak (left) Bruhat order is a convenient tool for studying $S_n$; one with applications in combinatorics as well as representation theory. The weak Bruhat order is defined as follows:
\[x \leq_{L} y \text{ if } x \text{ can be written as the suffix of a reduced expression for } y.\]
If we let $C(n,1) \coloneqq \{1,\ldots,n\}$, then we may identify $S_n$ with the set of total orders on $C(n,1)$ (see Definition \ref{DGH:def:Snbij}). Under this identification, the weak Bruhat order becomes a partial order on total orders. 

Meanwhile, any reduced expression $\rho$ for the longest word $w_0$ in $S_n$ can be viewed as a total order on pairs of strands. This is because each pair must cross once in $w_0$, and a reduced expression determines the order in which pairs of strands are crossed. Alternatively, $\rho$ induces a total order on
\[C(n,2) \coloneqq \{\text{size two subsets of } C(n,1)\}.\]
Manin and Schechtman discovered that there is a bijection between reduced expression for $w_0$ and \emph{admissible} total orders on $C(n,2)$ \cite[\S 2]{DGH:MS1989}. Manin and Schechtman then generalized the notion of admissibility to total orders on $C(n,k)$, the set of size $k$ subsets of the set with $n$ elements, calling them admissible $k$-orders.

By Matsumoto's theorem, any two reduced expressions are related by a sequence of braid moves. Braid moves come in two types: the \emph{commuting relation} $su = us$ whenever $m_{su} = 2$, and the \emph{Reidemeister III move} $sts = tst$ whenever $m_{st} = 3$. The analogous transformations of admissible $k$-orders are not too difficult to define combinatorically. The commuting relation correpsonds to \emph{elementary equivalences}, while the Reidemeister III move corresponds  to a \emph{packet flip}. Manin and Schechtman considered admissible $2$-orders up to elementary equivalence, which they view as the higher analog of $S_n$. They were then able to define a partial order on these equivalence classes. They called this partial order the second Bruhat order of $S_n$, and proved that it had a unique minimal and maximal element (alternately a \emph{source} and \emph{sink}). 

Manin and Schechtman went even further by defining the inversion set of a reduced expression/admissible $2$-order as a subset of $C(n,3)$, the size three subsets of $C(n,1)$. They found that maximal chains (ie paths from source to sink) in the second Bruhat order are admissible $3$-orders on $C(n,3)$. Continuing in this manner, Manin and Schechtman defined $B(n,k)$, the $k$-th Bruhat order of $S_n$, as a partial order on equivalence classes of admissible $k$-orders. 

The goal of this paper is to generalize Manin and Schechtman's results to non-longest elements. Showing that the reduced expression graph of $w$ has a source and sink is not new, see the discussion of \cite{DGH:BEDiam} below. However, studying paths from source to sink within the reduced expression graph of a non-longest word is new. So are placing a partial order on equivalence classes of these paths, proving that it has a source and sink, and iterating the procedure. Unlike in Manin and Schechtman's original paper, identifying the source and sink is subtle, see Theorem \ref{DGH:thm:DGHBruhat} for the answer.  

\subsection{Applications and other generalizations of higher Bruhat Orders}

Several uses of these higher Bruhat orders have been found. An extension of the second Bruhat order to non-reduced expressions was used by Ben Elias in \cite{DGH:BEDiam} to prove a Bergman diamond lemma for algebras with a presentation similiar to the Coxeter presentation of $S_n$, such as the (affine) Hecke algebra in type $A$ or Khovanov-Lauda-Rouquier algebras. Within the paper, the existence of a unique sink in $B(n,2)$ is used to prove that various ambiguities in words can be resolved in a sensible manner. The result is a convienent test for when certain monoidal categories defined by generators and relations are non-zero.

Another application of $B(n,2)$ comes from Elias' development of a thicker Soergel calculus \cite[\S1.3]{DGH:BESoerg}. Let $R$ be a polynomial ring in $n$ variables. For any simple reflection $s$ there is an $(R,R)$-bimodule $B_s$. To an expression expression $s_1\cdots s_l$ we can associate a tensor products $B_{s_1}\otimes\cdots\otimes B_{s_l}$, called a \emph{Bott-Samuelson bimodule}. The Bott-Samuelson bimodule of a reduced expression for $w \in S_n$ has a special indecomposable direct summand $B_{w}$. These are the indecomposable \emph{Soergel bimodules}. In general, computing the projection map from the Bott-Samuelson bimodule to the Soergel bimodule is quite difficult. However, it is possible to construct this idempotent for the reduced expressions of the longest element $w_0$ in $S_n$ (or in any parabolic subgroup).

When computing these idempotents, the commutation relations in the Coxeter presentation of $S_n$ correspond to isomorphisms between Bott-Samuelson bimodules. These isomorphisms form a compatible system, so one can associate a single Bott-Samuelson bimodule to an equivalence class in $B(n,2)$. The $m_{st} = 3$ braid relation corresponding to a Reidemeister III move corresponds to a morphism between Bott-Samuelsons which projects to a common summand. Because the Reidemeister III moves correspond to projections and not isomorphisms, it one cannot apply them indiscriminately. Elias proves that, by applying them in specific orders, one can obtain the projection to $B_w$. One such order comes from the second Bruhat order. 

The idempotents computed in \cite{DGH:BESoerg} are important when diagrammatically categorifying the Hecke algebra of type A. It is conjectured that higher Bruhat orders ($k > 2$) will play a role in higher categorifications of the Hecke algebra. 

\begin{rmk}
	It is an open question whether or not one can use similar methods to construct idempotents for Bott-Samuelson bimodules associated to reduced expressions of \textit{non-longest words} in $S_n$. These idempotents would no longer project to $B_{w}$, but the bimodules that they do project to could still be of interest. In particular, they would not depend upon the characteristic of the base field, whereas $B_w$ does. Such bimodules were first studied by Libedinsky, who explored them in extra-large Coxeter type \cite{DGH:LIB}.
\end{rmk}

There are also partial extensions of thicker Soergel calculus in type B. Koley \cite{DGH:Koley} extended Elias' thicker Soergel calculus to the Weyl group $B_3$. Likewise, there exist partial generalizations of higher Bruhat orders to type $B$ Weyl groups. Work by Shelley-Abrahamson and Vijaykumar \cite{DGH:SH2016} defined the second and third Bruhat orders for $B_n$. Generalizing to higher $k$ is difficult.

\begin{rmk}
	The author wrote code to compute higher Bruhat orders in Type A and Type B. Pseudocode for the type A algorithms is found in the appendix \ref{DGH:Sec:App}. The type B algorithms are very similar. In particular, it becomes extremely computationally intensive to produce type B examples as $B_n$ grows faster than $S_n$. This makes it difficult to experimentally determine what the correct partial order on total orders of the analog of $C(n,k)$ should be. Without an outside theory to guide exploration, we were unable to make progress on type B higher Bruhat orders.
\end{rmk} 

Another possibly fruitful avenue of research is extending higher Bruhat orders to affine type A Weyl groups. The most obvious difference between finite and affine type A is that affine Weyl groups lack longest elements. This makes it impossible to talk about maximal chains in the weak Bruhat graph of affine Type A. Instead, higher Bruhat orders must be defined for individual words. This has parallels to defining higher Bruhat orders for non-longest words in $S_n$, which is the topic of this paper. We explore the first steps towards an affine generalization in \S\ref{DGH:sec:Affine}.

\subsection{Paper overview}

The paper is divided into four main sections. After the introduction, Section \ref{DGH:Sec:One} provides a relatively in-depth exploration of Manin and Schechtman's original ideas. We omit any proof of Manin and Schechtman's main result (Theorem \ref{DGH:thm:MSthm}) and instead focus on examples and definitions of key concepts. Special attention is paid to Ziegler's isomorphism between higher Bruhat orders and the \emph{single-step inclusion order on realizable sets}(Theorem \ref{DGH:thm:Zieg}). This isomorphism establishes a correspondence between equivalence classes of admissible $k$-orders and realizable $(k+1)$-sets. Realizable sets and their complement obey nice convexity requirements and arise as inversion sets of admissible $k$-orders. See \cite{DGH:MS1989} and \cite{DGH:Z1993} for Manin and Schechtman, and Ziegler's original results, respectively.

Section \ref{DGH:Sec:Adm} lays the groundwork for defining higher Bruhat orders for non-longest words. In it, the necessary definitions and concepts from Ziegler as well as Manin and Schechtman's papers are extended to arbitrary realizable sets. The main technical result in the section is a bijection between paths from the empty set to a realizable $k$-set $J$ in a higher Bruhat graph with admissible orders on $J$. Ziegler's mapping between admissible $k$-orders and realizable $(k+1)$-sets is generalized to a map between admissible orders on $J$ and realizable $(k+1)$-sets.  The section concludes by defining the second Bruhat order for $J$ as the poset of equivalence classes of admissible orders on $J$. This is similar to how Manin and Schechtman define higher Bruhat orders for $S_n$ as posets of equivalence classes of admissible orders on $C(n,k)$. We also prove that the second Bruhat order for $J$ has a unique minimal and maximal element.

Section \ref{DGH:Sec:GenBruhat} opens by showing how the third Bruhat order of $J$ can be constructed out of maximal chains in the second Bruhat order of $J$. In order to do so, we exploit the fact that the third Bruhat order of $J$ can be embedded into the second Bruhat order of $S_n$. This is accomplished by extending maximal chains in the second Bruhat order of $J$ to certain chains in $B(n,2)$. Using special properties of the minimal and maximal elements of the second Bruhat order of $J$, we show that the resulting embedding is independent of the choice of extension. As one can tell, this is fairly technical. For clearer details refer to Section \ref{DGH:ssec:BPaths}.  

The remainder of Section \ref{DGH:Sec:GenBruhat} is devoted to proving Theorem \ref{DGH:thm:DGHBruhat}, our main result. Exploiting the correspondence between realizable $2$-sets and words $w$ in $S_n$ (this comes from taking the inversion set of $w$), higher Bruhat orders for non-longest words are defined. The $i$-th Bruhat order for $w$ is defined as a poset of equivalence classes of paths within $B(n,i-1)$ between realizable $i$-sets $L^i$ and $M^i$. The definitions of $L^i$ and $M^i$ depend on $w$ and are somewhat subtle (see Section \ref{DGH:ssec:BPaths} for details). In particular, they correspond to the inversion sets of the minimal and maximal elements of the $(i-1)$ Bruhat order of $J$. Defining them this way ensures that the $i$-th Bruhat order of $J$ can be created out of maximal chains in the $(i-1)$ Bruhat order, just as $B(n,i)$ is created out of maximal chains in $B(n,i-1)$.

The most difficult part of proving Theorem \ref{DGH:thm:DGHBruhat} is proving that $M^i$ is realizable and maximal. Both results are technical and crucially depend upon $J$ being a realizable $2$-set. The section concludes with a counterexample demonstrating that the techniques used each proof do not apply when $J$ is an arbitrary realizable $k$-set.

The main portion of the paper ends with preliminary forays into generalizing higher Bruhat orders to affine type A. It is followed by an appendix containing code for generating higher Bruhat orders.

\subsection{Acknowledgements}

I would like to acknowledge my advisor Ben Elias for all of his help throughout my graduate education. He suggested the topic for this paper and proved invaluable in helping me through many of the big ideas, as well as throughout the editing process. I also received summer support from Elias' NSF grant, DMS-1553032.

\section{The Theorem of Manin and Schechtman} \label{DGH:Sec:One}

This section provides an overview of Manin and Schechtman's construction of higher Bruhat orders as well as their main theorem \cite{DGH:MS1989}. We also cite Ziegler's work establishing the equivalence between Manin and Schechtman's higher Bruhat orders with Ziegler's own single step inclusion order \cite{DGH:Z1993}. Ultimately, we will work primarily with Ziegler's single step inclusion order as it is more amenable to direct analysis. The first half of this exposition closely follows the paper by Shelley-Abrahamson and Vijaykumar \cite{DGH:SH2016}. 

\subsection{Packets and admissible orders}

Given a finite set $X$ we define 
\[
C(X,k) \coloneqq \{x \subseteq X \mid \vert x \vert = k\}
\]
as the collection of all $k$-element subsets of $X$. The elements of $C(X,k)$ are called the \emph{$k$-sets of $X$}. Throughout the paper we will frequently work with $I_n = \{1,2,\dots,n\}$; the set of n distinct elements, totally ordered in the usual way. For convenience and in keeping with the literature, $C(n,k) \coloneqq C(I_n,k)$.

\begin{definition}\label{DGH:def:packet}
	For each $X \in C(n,k+1)$ we define the \emph{packet of X}, $P_X$, to be the collection of $k$-sets in $X$. In other words, $P_X = C(X,k)$. Any subset of $C(n,k)$ of the form $P_X$ for some $X \in C(n,k+1)$ is called a \emph{$k$-packet}. 
\end{definition}

\begin{example}[$2$-packet]\label{DGH:ex:2pack} Let $X = \{1,2,3\}$. Then $P_X$ is
	\begin{equation}
		P_X = \{\{1,2\}, \{1,3\}, \{2,3\}\}.
	\end{equation}
\end{example}

\begin{lemma}\label{DGH:lem:SharePacket} Let $J$ and $K$ be two $k$-sets. If $\vert J \cap K\vert = k-1$, then $J$ and $K$ are in exactly one $k$-packet together, namely $P_{J\cup K}$. Moreover, $P_J \cap P_K$ contains one $(k-1)$-set, $J \cap K$. If $\vert J\cap K\vert < k-1$, then $J$ are $K$ are not in any shared $k$-packet. Additionally, $P_J \cap P_K$ is empty. 
\end{lemma}
\begin{proof}
	The proof is evident and left to the reader. 
\end{proof}

\begin{notation}
	Given the frequency with which we work with sets of sets in this paper, it is convenient to drop some of traditional trappings of set notation. If $X = \{x_1,\ldots,x_{k}\}$, we will use $x_i$ and $i$ interchangeably. When the meaning is clear, we will also dispense with braces and commas; writing $\{x_1,\ldots x_{k}\}$ as $\{x_1\ldots x_{k}\}$ or $x_1\ldots x_{k}$. This notation is often used when writing down $P_X$. For example, both $\{1234\}$ and $1234$ represent $\{1,2,3,4\}$, and $P_{\{1,2,3,4\}}$ may be expressed as $P_{\{1234\}}$ or $P_{1234}$.
	
	We also use special notation for the elements of $P_X$. Each element of $P_X$ is of the form $X\sdiff{x_i}$ for $1 \leq i \leq k$. The notation $\widehat{x}_i$ and $\widehat{i}$ will be used interchangeably for $X\sdiff{x_i}$. The notation $\widehat{i}\widehat{j}$ denotes $X\sdiff{x_i,x_j}$, etc. 
\end{notation}

Because $I_n$ is totally ordered, $C(n,k)$ has a lexicographical total order, denoted $\rlex$, and an antilexicographical total order, denoted $\ranti$. Any $k$-packet inherits these total orders by restriction.

\begin{example}\label{DGH:ex:lexord} The packet $P_{1234} = \{123, 124, 134, 234\}$ is ordered lexicographically and antilexicographically as follows
	\begin{align*}
		\rho_{\lex}: & 123 < 124 < 134 < 234, \\
		\rho_{\anti}: & 234 < 134 < 124 < 123.
	\end{align*}
\end{example}

The remainder of this paper is primarily devoted to studying certain total orders on $C(n,k)$ and comparing them to the lexicographic order. Total orders on $C(n,1)$ can be identified with the symmetric group $S_n$. By defining an appropriate ``distance'' between an order and the lexicographic order (see Definition \ref{DGH:def:InvSets}) one can recover the length function and weak left Bruhat order on $S_n$. In order to accomplish this, we must impose an additional condition on our total orders which is trivial when $k = 1$. This condition is called \emph{admissibility}. 

\begin{definition}\label{DGH:def:adm}
A total order $\rho$ on $C(n,k)$ is called a \emph{$k$-order}. A $k$-order is \emph{admissible} (alt. admissible $k$-order) if $\rho$ induces either the lexicographical or antilexicographical ordering on every $k$-packet in $C(n,k)$.

We use $A(n,k)$ to denote the set of all admissible $k$-orders on $C(n,k)$.
\end{definition}

\begin{example}\label{DGH:ex:adm} Here we provide examples of admissible and non-admissible 2-orders on $C(4,2)$. We also list the induced orders on each two-packet.
	\begin{center}
		\begin{tabular}{c|c|c} 
			\multicolumn{1}{c}{}& \multicolumn{1}{c}{Admissible}& \multicolumn{1}{c}{Not Admissible} \\
			Order on $C(4,2)$ & $ 23 < 13 < 24 < 14 < 12 < 34 $& $12 < 13 < 34 < 14 < 24 < 23$ \\
			\hline
			Order on $\{123\}$ & $23 < 13 < 12 $ & $12 < 13 < 23$\\ 
			Order on $\{124\}$ & $24 < 14 < 12$ & $12 < 14 < 24$ \\
			Order on $\{134\}$ & $13 < 14 < 34$ & $13 < 34 < 14$ \\
			Order on $\{234\}$ & $23 < 24 < 34$ & $34 < 24 < 23$
		\end{tabular}
	\end{center}
	The second order fails admissibility because its restriction to $P_{134}$ is neither the lexicographic nor the antilexicographic order.
\end{example}

The set $A(n,k)$ is always non-empty as $\rlex$, the lexicographical ordering of $C(n,k)$, is always admissible. Likewise the antilexicographic order $\ranti$ is also always admissible. More generally, if $\rho$ is admissible, then its reverse $\rho^{t}$ is also admissible. 

\begin{notation}\label{DGH:not:rhot}
	Given an admissible $k$-order we denote the reverse order as $\rho^t$. 
\end{notation}

When $k = 1$, a total order on $1$-sets is the same as a total order on $I_n$. In this case, admissibility is trivial as a $1$-packet can only be ordered lexicographically or antilexicographically. Hence, $A(n,1)$ corresponds to the set of total orderings of $I_n$. Therefore $A(n,1)$ is also in bijection with the symmetric group $S_n$, where $\rlex$ corresponds to the identity element and $\ranti$ to $w_0$, the longest word in $S_n$. In this sense, admissible $k$-orders generalize the notion of a permutation in a more restrictive way than just considering total orders on $C(n,k)$. 

Because there are multiple natural bijections between $S_n$ and $A(n,1)$, we will fix a particular bijection for the rest of this paper. 

\begin{definition}\label{DGH:def:Snbij}
	Let $\iota: S_n \rightarrow A(n,1)$ send a permutation $w \in S_n$ to the total order on $I_n$ defined by 
	\begin{equation}\label{DGH:eqn:Snbij}
		\iota(w): w^{-1}(1) < w^{-1}(2) < \ldots < w^{-1}(n-1) < w^{-1}(n)
	\end{equation}
\end{definition}

If admissible $k$-orders are meant to generalize permutations, then we should be able to generalize the length function of $S_n$. 

\begin{definition}\label{DGH:def:inv}
	Take $\rho \in A(n,k)$. The \emph{inversion set of } $\boldsymbol{\rho}$ is defined as
\begin{equation}\label{DGH:def:InvSets}
	\Inv(\rho) \coloneqq \{X \in C(n,k+1) \mid \rho\vert_{ P_X} = \ranti\vert_{P_X}\}.
\end{equation}
We may view $\Inv$ as a function $\Inv:A(n,k) \rightarrow 2^{C(n,k+1)}$, where $2^{C(n,k+1)}$ is the power set of $C(n,k+1)$. The \emph{length of $\rho$} is defined as $l(\rho) \coloneqq \vert\Inv(\rho)\vert$.	
\end{definition}

\begin{rmk}\label{DGH:rmk:InvAgree}
	The inversion set of a permutation $w \in S_n$ is defined as 
	\begin{equation}
		\Inv(w) \coloneqq \{(x,y) \in I_n \vert x < y \text{ and} w(x) < w(y)\}
	\end{equation}
	The length function of $w$ can then be defined as $l(w) = \vert \Inv(w)\vert$. It is worth checking that $\Inv(w) = \Inv(\iota(w))$ where $\iota: S_n \rightarrow A(n,1)$ is the bijection defined in Equation \eqref{DGH:eqn:Snbij}
\end{rmk}

\begin{example}{(Inversion sets of $S_n$)}\label{DGH:ex:SnInv} Let $S_3 = \langle s = (12), t = (23) \mid s^2 = t^2 = 1, sts = tst\rangle$. Under $\iota$,  
	\[ts \mapsto \ranti: 2 < 3 < 1\]
	Here $\Inv(ts) = \{12, 13\}$. Furthermore, $l(ts) = 2$. 
\end{example}

\begin{example}\label{DGH:ex:Inv} The inversion set for the admissible 2-order in Example \ref{DGH:ex:adm} is 
	\[\Inv(\rho) = \{123, 124\}\]
\end{example}

\begin{lemma}\label{DGH:lem:InvRev} Let $\rho \in A(n,k)$, then $\Inv(\rho^t) = \Inv(\rho)^c$. Here $(\cdot)^c$ denotes the complement of a set. 
\end{lemma}
\begin{proof}
	The involution $(\cdot)^t$ reverses the order on every $k$-packet. If $\rho$ induced the lexicographic order on $P_x$, then $\rho^t$ induces the antilexicographic order on $P_x$ and vice versa.
\end{proof}

Just as the length function on $S_n$ captures how far away an element $w \in S_n$ is from the identity element in the weak left Bruhat order, $\vert\Inv(\cdot)\vert$ will capture how far an admissible $k$-order is from $\rlex$ in a more general version of the weak left Bruhat order. Formalizing this notion requires a bit more work.

\subsection{Reduced expressions and admissible $2$-orders}

One of the more interesting consequences of Manin and Schechtman's theorem is that admissible $2$-orders on $C(n,2)$ are in bijection with the reduced expressions for the longest word $w_{0} \in S_n$. We will delay the proof of this fact until after Manin and Schechtman's theorem has been stated. In the meantime, understanding this correspondence when $n = 4$ will help motivate many of the subsequent definitions.

Recall the Coxeter presentation of $S_4$ 
\begin{equation*}\label{DGH:eq:S4Cox}
	S_4 = \langle s, t, u \mid s^2 = t^2 = u^2 = 1, su = us, sts = tst, utu = tut \rangle,
\end{equation*}
where $s = (12)$, $t = (23)$, $u = (34)$ are the simple transpositions swapping adjacent elements in a list. The longest word $w_0$ maps $1 \mapsto 4$, $2 \mapsto 3$, $3 \mapsto 2$, $4 \mapsto 1$. In order for this to occur, $w_0$ must invert every pair $i < j$. Hence $\Inv(w_0) = C(4,2)$. Each reduced expression for $w_0$ fixes the order in which these inversions occur. Put another way, each reduced expression places an order on $C(4,2)$ based on when the $1$-packet $\{ij\}$ is flipped. A picture is worth a thousand words.

\begin{example}\label{DGH:ex:stutst}
	Consider the string diagram for the reduced expression $w_0 = stutst$ (see Figure \ref{DGH:fig:stutst}).
	\begin{figure}
		\begin{tikzpicture}
			\pic[braid/floor 1/.style={fill=white},
			braid/floor 2/.style={fill=gray!50},
			braid/floor 3/.style={fill=white},
			braid/floor 4/.style={fill=gray!50},
			braid/floor 5/.style={fill=white},
			braid/floor 6/.style={fill=gray!50},
			line width = 2pt,
			braid/strand 1/.style={red},
			braid/strand 2/.style={blue},
			braid/strand 3/.style={green},
			braid/strand 4/.style={purple},
			name prefix = braid
			]{braid = { s_1  s_2  s_3  s_2  s_1  s_2}};
			\node[circle,draw] (c1) at (.5,-.75){};
			\node[circle,draw] (c2) at (1.5,-1.75){};
			\node[circle,draw] (c3) at (2.5,-2.75){};
			\node[circle,draw] (c4) at (1.5,-3.75){};
			\node[circle,draw] (c5) at (.5,-4.75){};
			\node[circle,draw] (c6) at (1.5,-5.75){};
			\node[] at (3.5,-.75){\{34\}};
			\node[] at (3.5,-1.25){$\vee$};
			\node[] at (3.5,-1.75){\{24\}};
			\node[] at (3.5,-2.25){$\vee$};
			\node[] at (3.5,-2.75){\{14\}};
			\node[] at (3.5,-3.25){$\vee$};
			\node[] at (3.5,-3.75){\{12\}};
			\node[] at (3.5,-4.25){$\vee$};
			\node[] at (3.5,-4.75){\{13\}};
			\node[] at (3.5,-5.25){$\vee$};
			\node[] at (3.5,-5.75){\{23\}};
			\node[] at (0,.25) {4};
			\node[] at (1,.25){3};
			\node[] at (2,.25){2};
			\node[] at (3,.25){1};
			\node[] at (0,-6.75){1};
			\node[] at (1,-6.75){2};
			\node[] at (2,-6.75){3};
			\node[] at (3,-6.75){4};
		\end{tikzpicture} 
	\caption{String diagram for $w = stutst$ used in Example \ref{DGH:ex:stutst}}
	\end{figure}
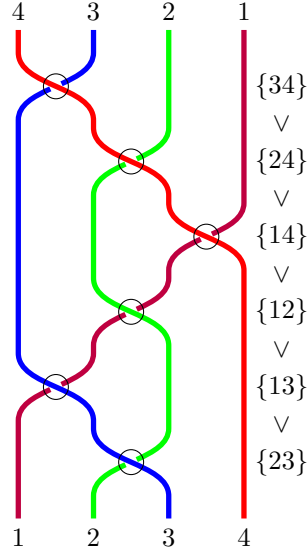\label{DGH:fig:stutst}
	Reading bottom to top, this reduced expression produces the $2$-order 
	\[\rho: \{23\} < \{13\} < \{12\} < \{14\} < \{24\} < \{34\}\]
	The reader should verify that this is indeed an admissible $2$-order.
\end{example}
In general, the $2$-order corresponding to a reduced expression will be admissible. This is because a $2$-packet generated by $\{i<j<l\}$ can be used to keep track of when the $i$, $j$, and $l$ strands crossed each other. Out of the three crossings $ij, il, jl$ only $il$ cannot occur first, because the middle strand $j$ needs to be moved out of the way. If $i$ crosses $j$ first, then $j$ cannot cross $l$ until $i$ is moved past $l$. So the order of crossings must be $ij < il < jl$. Similarly, if $j$ crosses $l$ first, then crossings must occur in the antilexicographic order.

\begin{example}\label{DGH:ex:S4Cor} Here are a few more reduced expressions for the longest word in $S_4$ along with their accompanying admissible $2$-orders and inversion sets.
		\begin{center}
		\begin{tabular}{c|c|c} 
			\multicolumn{1}{c}{Rex}& \multicolumn{1}{c}{Order}&\multicolumn{1}{c}{$\Inv(\cdot)$}\\
			\hline 
			 tstuts & $12 < 13 < 14 < 34 < 24 < 23$ & $\{234\}$ \\
			 tsutus & $12 < 34 < 14 < 13 < 24 < 23$ & $\{134, 234\}$ \\
			 tustus & $12 < 34 < 14 < 24 < 13 < 23$ & $\{134, 234\}$ \\
			 tustsu & $34 < 12 < 14 < 24 < 13 < 23$ & $\{134, 234\}$ \\
			 tutstu & $34 < 24 < 14 < 12 < 13 < 23$ & $\{124, 134, 234\}$ \\
		\end{tabular}
	\end{center}
\end{example}

\begin{rmk}
	Recall that $\Inv(\cdot)$ is an injective map on $S_n = A(n,1)$ into the powerset of $C(n,2)$. Example \ref{DGH:ex:S4Cor} shows that $\Inv(\cdot)$ is no longer injective on $A(n,2)$. Some reduced expressions generate admissible $2$-orders with the same inversion sets! In Example \ref{DGH:ex:S4Cor}, the three reduced expressions with the same inversion sets all differed by a sequence of commuting relations, such as $su = us$. This can be generalized and will motivate the definition of an \emph{elementary equivalence} in Section \ref{DGH:sub:ElemEquiv}.
	
	The reduced expressions $tstuts$ and $tsutus$ have inversion sets that differ by a single element. Additionally, $tstuts$ and $tsutus$ are related by the Reidemeister III move $tut = utu$. In general, two reduced expressions related by a Reidemeister III move will generate inversion sets which differ by a single $3$-set. This observation will motivate the definition of a \emph{packet flip}. 
\end{rmk}

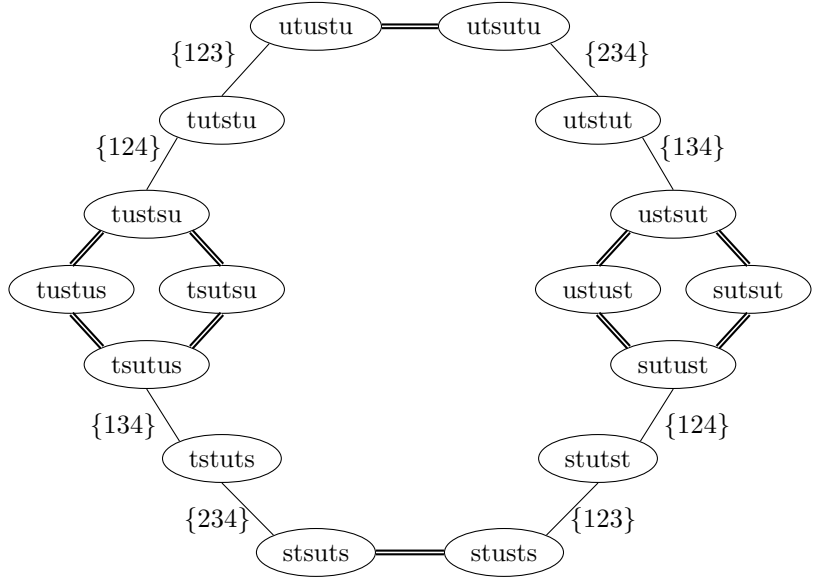
\begin{figure}
	\centering
	\begin{tikzpicture}
		\node[ellipse, draw] (6R) at (5.25,7){utsutu};
		\node[ellipse, draw] (6L) at (2.75,7){utustu};
		\node[ellipse, draw] (5R) at (6.5,5.75){utstut};
		\node[ellipse, draw] (5L) at (1.5,5.75){tutstu};
		\node[ellipse, draw] (4R) at (7.5,4.5){ustsut};
		\node[ellipse, draw] (4L) at (.5,4.5){tustsu};
		\node[ellipse, draw] (3FR) at (8.5,3.5){sutsut};
		\node[ellipse, draw] (3R) at (6.5,3.5){ustust};
		\node[ellipse, draw] (3L) at (1.5,3.5){tsutsu};
		\node[ellipse, draw] (3FL) at (-.5,3.5){tustus};
		\node[ellipse, draw] (2R) at (7.5,2.5){sutust};
		\node[ellipse, draw] (2L) at (.5,2.5){tsutus};
		\node[ellipse, draw] (1R) at (6.5,1.25){stutst};
		\node[ellipse, draw] (1L) at (1.5,1.25){tstuts};
		\node[ellipse, draw] (0L) at (2.75,0){stsuts};
		\node[ellipse, draw] (0R) at (5.25,0){stusts};
		
		\draw[double,thick] (6L.east) -- (6R.west);

		\draw[-] (6L.south west) -- (5L.north) node[above, xshift = -.25cm, yshift = .25cm]{\{123\}};
		\draw[-] (6R.south east) -- (5R.north) node[above, xshift = .25cm, yshift = .25cm]{\{234\}};
		
		\draw[-] (5L.south west) -- (4L.north) node[above, xshift = -.25cm, yshift = .25cm]{\{124\}};
		\draw[-] (5R.south east) -- (4R.north) node[above, xshift = .25cm, yshift = .25cm]{\{134\}};
		
		\draw[double, thick] (4L.south west) -- (3FL.north);
		\draw[double, thick] (4L.south east) -- (3L.north);
		\draw[double,thick] (4R.south east) -- (3FR.north);
		\draw[double,thick] (4R.south west) -- (3R.north);
		
		\draw[double, thick] (3FL.south) -- (2L.north west);
		\draw[double, thick] (3L.south) -- (2L.north east);
		\draw[double,thick] (3FR.south) -- (2R.north east);
		\draw[double,thick] (3R.south) -- (2R.north west);
		
		\draw[-] (2L.south) -- (1L.north west) node[below, xshift = -.75cm, yshift = .5cm]{\{134\}};
		\draw[-] (2R.south) -- (1R.north east) node[below, xshift = .75cm, yshift = .5cm]{\{124\}};
		
		\draw[-] (1L.south) -- (0L.north west) node[below, xshift = -.75cm, yshift = .5cm]{\{234\}};
		\draw[-] (1R.south) -- (0R.north east) node[below, xshift = .75cm, yshift = .5cm]{\{123\}};
		
		\draw[double,thick] (0L.east) -- (0R.west);
	\end{tikzpicture}
\caption{The reduced expression graph for the longest word $w_0$ in $S_4$. Each node is a different reduced expression. Two reduced expressions are linked by a double edge if they differ by a commuting relation. Reduced expressions are linked by a single edge if they differ by a Reidemeister III move. Reading bottom to top, each single edge is labeled by the $3$-set that is added to the inversion set of the admissible $2$-order for the resulting reduced expression.}
\end{figure}\label{DGH:fig:S4rex}

\subsection{Elementary equivalences and packet flips}\label{DGH:sub:ElemEquiv}

Recall from Example $\ref{DGH:ex:S4Cor}$ that $\Inv(\cdot)$ is not injective on $A(n,2)$. It will not be injective on $A(n,k)$ for any $k > 1$ unless $n$ is small. The set $A(n,k)$ is too large. In order to correct this, we define an equivalence relation on $A(n,k)$ which is analogous to the commuting relations in the Coxeter presentation of $S_n$ (like $su = us$). The function $\Inv(\cdot)$ will descend to an injective function on these equivalence classes (see Theorem \ref{DGH:thm:MSthm}(\ref{MSthm:partfour})). 

\begin{definition}\label{DGH:def:equiv}
Two $k$-orders $\rho, \rho' \in A(n,k)$ are said to be \emph{elementary equivalent} if they differ up to an interchange of two neighboring elements that are not in a shared $k$-packet.
\end{definition}

\begin{example}\label{DGH:ex:equiv} When $k = 1$, no two admissible $k$-orders are elementary equivalent to each other. This is because every pair of distinct elements $i,j \in I_n$ (viewed as sets of size 1) are in the shared $1$-packet $P_{ij}$.
	
When $k > 1$, a pair of $k$-sets $X, Y \in C(n,k)$ are found in the same $k$-packet if and only if $\vert X \cap Y\vert = k - 1$, see Lemma \ref{DGH:lem:SharePacket}. This allows elementary equivalent admissible $k$-orders to exist. For example,
	\begin{align*}
		\rho &: 12 < 13 < 14 < 23 < 24 < 34 \\
		\rho' &: 12 < 13 < 23 < 14 < 24 < 34
	\end{align*}
	are elementary equivalent.
\end{example}

\begin{definition}\label{DGH:def:rel}
Define $\sim$ to be the equivalence relation generated by these elementary equivalences. So $\rho \sim \rho'$ if and only if they can be obtained from one another by a sequence of elementary equivalences. 
\end{definition}

Elementary equivalences permute the order $\rho$ by a simple transposition, swapping two adjacent term in $\rho$. Thus a sequence of elementary equivalences can be viewed as an expression for a permutation in $S_N$, where $N$ is the size of $C(n,k)$. Of course, not all permutations in $S_N$ correspond to a sequence of elementary equivalences, as we are only allowed to swap adjacent terms which do not appear in the same packet. Regardless, we can use some of the same concepts which apply to expressions of permutations when we study elementary equivalences.

\begin{lemma}\label{DGH:lem:OverIt}
	Let $\rho$ be an admissible $k$-order. Let $N$ be the size of $C(n,k)$, and identify permutations of $C(n,k)$ with $S_N$ using the total ordering given by $\rho$. Suppose that $\rho$ and $\rho'$ are equivalent, where $w$ is the permutation of $C(n,k)$ such that $w(\rho) = \rho'$. Then any sequence of elementary equivalences morphing $\rho$ into $\rho'$ is an expression for $w$ in $S_N$. Conversely, any reduced expression for $w$ will give a sequence of elementary equivalences taking $\rho$ to $\rho'$.
\end{lemma}
\begin{proof}
	The only non-trivial statement is made in the final sentence. Because $\rho$ and $\rho'$ are equivalent to each other, we only need to permute terms of $\rho$ which are not in the same packet in order to transform $\rho$ into $\rho'$. Therefore any expression for $w$ which involves swapping terms in a shared packet must undo that swap. So any such expression is not reduced. Necessarily, any reduced expression gives a sequence of elementary equivalences taking $\rho$ to $\rho'$. 
\end{proof}

\begin{definition}\label{DGH:def:Bnk}
	The set $B(n,k)$ is defined as the quotient of $A(n,k)$ by $\sim$. For any $\rho \in A(n,k)$, we denote its class in $B(n,k)$ as $[\rho]$. The set $B(n,k)$ is called the \emph{$k$-th higher Bruhat order of $S_n$}.
\end{definition} 

\begin{example}
	A graph of $B(4,2)$ is obtained from Figure \ref{DGH:fig:S4rex} by collapsing any nodes connected by a double edge. See Figure \ref{DGH:fig:B42} for the actual graph.
\end{example}

\begin{figure}
	\centering
	\begin{tikzpicture}
		\node[ellipse, draw] (4) at (4,6){\{utsutu, utustu\}};
		\node[ellipse, draw] (3R) at (6.5,4.75){utstut};
		\node[ellipse, draw] (3L) at (1.5,4.75){tutstu};
		\node[ellipse, draw, align=left] (2R) at (7.5,3){\{ustsut, ustust, \\ sutsut, sutust\}};
		\node[ellipse, draw, align=left] (2L) at (.5,3){\{tustsu,tustus, \\ tsutsu, tsutus\}};
		\node[ellipse, draw] (1R) at (6.5,1.25){stutst};
		\node[ellipse, draw] (1L) at (1.5,1.25){tstuts};
		\node[ellipse, draw] (0) at (4,0){\{stsuts,stusts\}};
		
		
		\draw[-] (4.south west) -- (3L.north) node[above, xshift = -.25cm, yshift = .25cm]{\{123\}};
		\draw[-] (4.south east) -- (3R.north) node[above, xshift = .25cm, yshift = .25cm]{\{234\}};
		
		\draw[-] (3L.south west) -- (2L.north) node[above, xshift = -.25cm, yshift = .25cm]{\{124\}};
		\draw[-] (3R.south east) -- (2R.north) node[above, xshift = .25cm, yshift = .25cm]{\{134\}};
		
		\draw[-] (2L.south) -- (1L.north west) node[below, xshift = -.75cm, yshift = .5cm]{\{134\}};
		\draw[-] (2R.south) -- (1R.north east) node[below, xshift = .75cm, yshift = .5cm]{\{124\}};
		
		\draw[-] (1L.south) -- (0.north west) node[below, xshift = -1.25cm, yshift = .5cm]{\{234\}};
		\draw[-] (1R.south) -- (0.north east) node[below, xshift = 1.25cm, yshift = .5cm]{\{123\}};
	\end{tikzpicture}
	\caption{Graph of $B(4,2)$}
\end{figure}
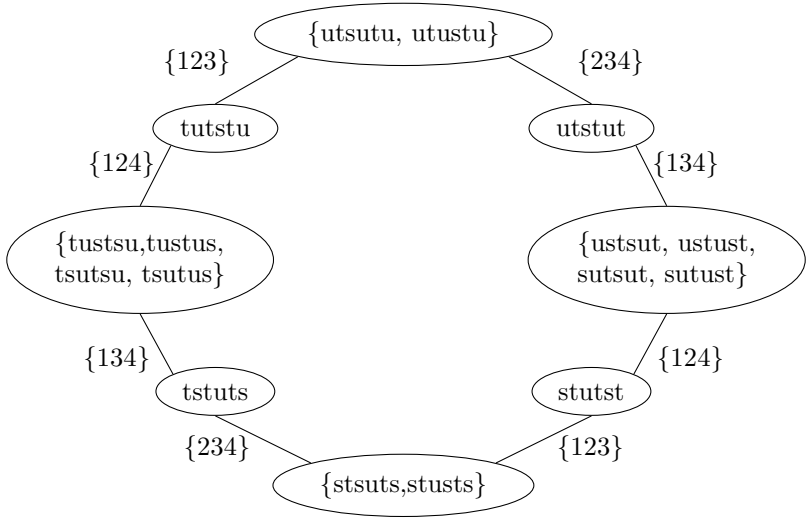\label{DGH:fig:B42}

Since the neighboring elements are not in a shared packet, two elementary equivalent orders have the same inversion set. Therefore $\Inv\text{: } A(n,k) \rightarrow 2^{C(n,k+1)}$ descends to a well-defined function on $B(n,k)$. Let $\Inv([\rho]) \coloneqq \Inv(\rho)$ for any $\rho \in [\rho]$. For the remainder of this paper we reserve the letter $r$ for elements of $B(n,k)$. For convenience, we follow \cite{DGH:SH2016} and let $r_{\min} = [\rlex]$, $r_{\max} = [\ranti]$. The use of ``max'' and ``min'' will be justified shortly. 

Elementary equivalences leave the inversion set of an admissible $k$-order unchanged. We now introduce an operation on admissible $k$-orders which alters the inversion set. When $k = 2$, we will see that this operation generalizes the Reidemeister III moves $sts = tst$ and $utu = tut$ seen before in $S_3$ and $S_4$. 

If the elements of some $k$-packet $P_X$ form a contiguous block in an order $\rho$ then we say that $P_X$ forms a \emph{chain} in $\rho$. Using this we define the function $N: A(n,k) \rightarrow 2^{C(n,k+1)}$  as
\begin{equation}N(\rho) \coloneqq \{X \in C(n,k+1) \mid P_X \text{ forms a chain in } \rho\}\end{equation}
If $X \in N(\rho)$ then we may ``flip'' $P_X$ in $\rho$ to create a new admissible $k$-order. This is done by reversing the order on $P_X$. 
\begin{definition}\label{DGH:def:flip}
	Let $\rho \in A(n,k)$ and $X \in N(\rho)$. Define $p_X(\rho)$ as the total order in which the chain $P_X$ is reversed while all the other elements remain the same. We call this operation a \emph{packet flip} and denote the resulting order as $p_X(\rho)$.
\end{definition}



\begin{example}\label{DGH:ex:S4Flip} The admissible $2$-orders 
	\begin{align*}
		\rho: 12 < 13 < 14 < 34 < 24 < 23 \\
		\rho': 12 < 34 < 14 < 13 < 24 < 23
	\end{align*}
are related by the packet flip of $134$. Recalling Example \ref{DGH:ex:S4Cor}, $\rho$ comes from the reduced expression $tstuts$, while $\rho'$ comes from $tsutus$. The packet flip of $134$ corresponds to the Reidemeister III move $tut = utu$.
\end{example}

Examples \ref{DGH:ex:S4Cor} and \ref{DGH:ex:S4Flip} not only demonstrate how packet flips generalize Reidemeister III moves. They also suggest how a packet flip increases or decreases the size of an inversion set. 

\begin{lemma}\label{DGH:lem:AdmFlip} Pick $\rho \in A(n,k)$ and $X \in N(\rho)$. Then $p_X(\rho)$ is admissible. Furthermore, the inversion set of $p_X(\rho)$ can be expressed as
	\upshape \begin{equation}\Inv(p_X(\rho)) = \left\{\begin{array}{cc}
			\Inv(\rho) \cup \{X\}, & \text{ if } X \in \Inv(\rho) \\
			\Inv(\rho)\sdiff{X}, & \text{ if } X \notin \Inv(\rho)
		\end{array} \right.  \label{DGH:eqn:flipinv}\end{equation}
\end{lemma}
\begin{proof}
	The admissibility of $p_X(\rho)$ follows from Lemma $\ref{DGH:lem:SharePacket}$. Reversing the order on $P_X$ will in no way alter the ordering on any other $k$-packet. Equation $\ref{DGH:eqn:flipinv}$ also follows from this observation. If $X \notin \Inv(\rho)$, then reversing the order on $P_X$ results in $p_X(\rho)$ inducing the antilexicographical ordering on $P_X$. Alternatively, if $\rho$ induces the antilexicographic order on $P_X$, then $p_X(\rho)$ will induce the lexicographic order.
\end{proof}

Analogous to $N(\cdot)$, we define a function $N\text{: } B(n,k) \rightarrow 2^{C(n,k+1)}$. It is defined as 
\begin{equation}N(r) \coloneqq \bigcup_{\rho \in r} N(\rho).\end{equation}
Packet flips are also well-defined on $B(n,k)$. Whenever $X \in N(r)$ there is at least one $\rho \in r$ for which $X \in N(\rho)$. We define $p_X(r) \coloneqq [p_X(\rho)]$. It is not obvious that this is well-defined. Perhaps $\rho_1$ and $\rho_2$ are representatives of $r$ for which $X$ is flippable, but $p_X(\rho_1)$ is not equivalent to $p_X(\rho_2)$. The following lemma rules out this possibility. First we give an illustrative example.

\begin{example}\label{DGH:ex:FlipREx}
	Let $P_X = \{\widehat{x}_k,\ldots, \widehat{x}_1\}$ be a $(k-1)$-packet. Suppose that $X \in N(\rho)$ for an admissible $(k-1)$-order $\rho$, and that $\rho$ induces the lexicographic order on $P_X$. Further suppose that $y \leq_{\rho} \widehat{x}_k$ are adjacent in $\rho$ and that $y$ is not in the same packet with any $\widehat{x_i}$. We have the following commuting square:
	\[\begin{tikzcd}	
		y\widehat{x}_k\ldots\widehat{x}_1 \arrow{r}{\sim} \arrow{d}{p_X(\rho)} & \widehat{x}_k y\ldots\widehat{x}_1 \arrow{r}{\sim} & \ldots \arrow{r}{\sim} & \widehat{x}_k\ldots\widehat{x}_1 y \arrow{d}{p_X(\rho)} \\ 
		y\widehat{x}_1\ldots\widehat{x}_k \arrow{r}{\sim} & \widehat{x}_1 y\ldots\widehat{x}_k \arrow{r}{\sim} & \ldots \arrow{r}{\sim} & \widehat{x}_1\ldots\widehat{x}_k y 		
	\end{tikzcd}\]
\end{example}

\begin{lemma}\label{DGH:lem:FlipR} Let $X = \{x_1,\ldots,x_{k+1}\}$ and assume that $X \in N(r)$. The operation $p_X(r)$ is independent of the choice of the representative $\rho$.
\end{lemma}
\begin{proof}
	We assume that $k > 1$ as no two admissible $1$-orders are elementary equivalent to each other. Let $\rho,\rho' \in r$ such that $P_X$ is a chain in both $\rho$ and $\rho'$. We induct on $m$, the minimum number of elementary equivalences needed to transform $\rho$ into $\rho'$. In particular we can choose a sequence of elementary equivalences which form a reduced expression, in the sense of Lemma \ref{DGH:lem:OverIt}. 
	
	When $m = 1$, the elementary equivalence involves no elements in $P_X$. This is because $\vert P_X\vert > 1$ and $P_X$ forms a chain in both $\rho_1$ and $\rho_2$. In this case, $p_X(\rho) \sim p_X(\rho')$ by the same elementary equivalence as $\rho$ and $\rho'$.
	
	Assume that $m > 1$. Now there exists a sequence of elementary equivalences connecting $\rho$ with $\rho'$, $\rho = \rho_1 \sim_1 \ldots \sim_m \rho_{m+1} = \rho'$. Suppose that $\sim_1$ swaps two elements which are not in $P_X$. Then $P_X$ is still a chain in $\rho_2$. By induction, $[p_X(\rho)] = [p_X(\rho_2)] = [p_X(\rho')]$. 
	
	So we may assume, without loss of generality, that $\sim_1$ swaps an element $a \in C(n,k)\setminus P_X$ with $\widehat{x}_1$. Since our sequence of elementary equivalences form a reduced expression, $a$ will not be swapped with $\widehat{x}_1$ again. Since $P_X$ is consecutive in both $\rho$ and $\rho'$, $a$ must eventually be swapped with $\widehat{x}_i$ for all $i$. Then there is some reduced expression which swaps $a$ with all of $P_X$ first. Let $\rho''$ be the result of these first $k+1$ swaps. Clearly $X$ is flippable in $\rho''$ as well. Inspection reveals that $p_X(\rho) \sim p_X(\rho'')$, just as in Example \ref{DGH:ex:FlipREx}. Now $p_X(\rho'') \sim p_X(\rho')$ by induction (or $m = k+1$ and $\rho'' = \rho$).
	
	The minimality of $m$ along with the fact that $P_X$ is a chain in $\rho'$ demands that $a$ is swapped past all of $P_X$. So we may assume that $m \geq k+1$. If $m = k+1$ it is evident that $[p_X(\rho)] = [p_X(\rho')]$. If $m > k+1$, we may further assume that $a$ is commuted past $P_X$ by consecutive swaps. Induction gives $[p_X(\rho)] = [p_X(\rho_{k+2})] = [p_X(\rho')]$. 
\end{proof}

\begin{notation}
	When discussing a sequence of packet flips each of which increases the size of an inversion set, it is often cumbersome to repeatedly write $p_X(r)$ for $X \in N(r)\setminus \Inv(r)$. Therefore we will write $\flipup{r}{r'}{X}$ in place of $r' = p_{X}(r)$. Here it is assumed that $X \in N(r)\setminus \Inv(r)$.
\end{notation}

\begin{example}\label{DGH:ex:N(r)}
	Computing $N(r)$ and, to a lesser extent, $\Inv(r)$ can become quite cumbersome for large $n$ and $k$. An algorithm is provided in Appendix \ref{DGH:Sec:App}. Here we provide a worked out example in $B(4,2)$, corresponding to the equivalence class of the reduced expression $tustsu$.
	\begin{center}
		\begin{tabular}{c|c|c} 
			\multicolumn{1}{c}{Orders}& \multicolumn{1}{c}{$N(\cdot)$}&\multicolumn{1}{c}{$\Inv(\cdot)$}\\
			\hline $34 < 12 < 14 < 24 < 13 < 23$ & $\{124\}$ & $\{234, 134\}$\\
			$34 < 12 < 14 < 13 < 24 < 23$ & $\emptyset$& $\{234, 134\}$\\
			$12 < 34 < 14 < 24 < 13 < 23$ & $\emptyset$ & $\{234, 134\}$\\
			$12 < 34 < 14 < 13 < 24 < 23$ & $\{134\}$ & $\{234, 134\}$\\
			\hline r & $\{134,124\}$ & $\{234, 134\}$
		\end{tabular}
	\end{center}  
	
	Since $N(r) = \{134, 124\}$ we may perform two different packet flips on $r$. The result of those flips are contained in the following table:
	\begin{center}
		\begin{tabular}{c|c|c} 
			\multicolumn{1}{c}{Packet}& \multicolumn{1}{c}{Representative of $p_X(r)$}&\multicolumn{1}{c}{$\Inv([p_X(r)])$}\\
			\hline $\{12,14,24\}$ & $34 < 24 < 14 < 12 < 13 < 23$ & $\{234, 134, 124\}$\\
			$\{13,14,34\}$ & $12 < 13 < 14 < 34 < 24 < 23$& $\{234\}$
		\end{tabular}
	\end{center}                                                                                                               
\end{example}

\subsection{Manin-Schechtman Theorem}\label{DGH:MSthm:MSThm}

Manin and Schechtman proved the following theorem about the collections $A(n,k)$ and $B(n,k)$\cite[\S2]{DGH:MS1989}.

\begin{theorem}[MS, 1989]\label{DGH:thm:MSthm} Let $m \coloneqq \binom{n}{k}$. Define the following relation on $B(n,k)$
	\begin{equation*}\label{DGH:rel:MSrel}
		\begin{aligned}
			r \leq_{\MS} r' \Leftrightarrow &\text{ there exists a sequence of } X_i \in C(n,k+1) \text{ such that }\\			
			&\flipup{\flipup{r'=r_1}{r_2}{X_1}}{}{X_2} \ldots \flipup{}{\flipup{r_i}{r_{i+1} = r'}{X_i}}{X_{i-1}}
		\end{aligned}
	\end{equation*}
	Then the following are true:
	\begin{enumerate}
		
		\item The relation $\leq_{\MS}$ defines a ranked partial order on $B(n,k)$. The rank function is given by $\vert\Inv(\cdot)\vert$, \label{MSthm:partone}
		
		\item As a ranked poset, $B(n,k)$ has a unique minimal element $r_{\min}$ and unique maximal element $r_{\max}$. Their inversion sets are $\Inv(r_{\min}) = \emptyset$ and $\Inv(r_{\max}) = C(n,k+1)$. \label{MSthm:parttwo}
		
		\item \label{MSthm:MSCor} There is a bijection between 
		\[
		\{\text{the set of maximal chains in } B(n,k)\} \xrightarrow{\sim} A(n,k+1)
		\]
		given by the map
		\[
		\flipup{\flipup{r_{\min}=r_0}{r_1}{X_1}}{}{X_2} \ldots \flipup{}{\flipup{r_{m-1}}{r_{m} = r_{\max}}{X_{m}}}{X_{m-1}} \mapsto \rho \coloneqq X_1\ldots X_m
		\]
		
		This is called the \emph{Manin-Schechtman correspondence}.
		
		\item Each element of $B(n,k)$ is uniquely defined by its inversion set. \label{MSthm:partfour}
		
	\end{enumerate}
\end{theorem}

\begin{corollary}\label{DGH:cor:Bruhat}
	For fixed $n$, the poset $B(n,1)$ is isomorphic to $S_n$ as a poset under the weak left Bruhat order.
\end{corollary}
\begin{rmk}
	This result was originally observed by Manin and Schechtman, but a proof can be found in \cite[\S2]{DGH:SH2016}
\end{rmk}

\begin{corollary}\label{DGH:cor:RexGraph}
	For fixed $n$, the set $A(n,2)$ is in bijection with the set of reduced expressions of the longest word $w_0 \in S_n$. 
\end{corollary}
\begin{proof}
	By Corollary \ref{DGH:cor:Bruhat}, $B(n,1)$ corresponds to the usual weak left Bruhat order on $S_n$. In this case $1 \mapsto r_{\min}$ and $w_0 \mapsto r_{\max}$. Maximal chains in $B(n,1)$ therefore correspond to reduced expressions for the longest word $\omega_0 \in S_n$. Applying Theorem \ref{DGH:thm:MSthm}.\ref{MSthm:partfour} along with Matsumoto's theorem, we conclude that $B(n,2)$ is isomorphic to the reduced expression graph of $\omega_0$ in $S_n$.
\end{proof}

\begin{rmk}
	Given $\rho \in r$, we know that $\rho^t$ is an admissible $k$-order. We define $r^t \coloneqq [\rho^t]$. The quickest way to see that this is well-defined is to use Theorem \ref{DGH:thm:MSthm}. Note that if $\rho_1,\rho_2 \in r$, then $\rho_1^t$ and $\rho_2^t$ have the same inversion set by Lemma \ref{DGH:lem:InvRev}. Theorem \ref{DGH:thm:MSthm}(\ref{MSthm:partfour}) then states that $[\rho_1^t] = [\rho_2^t]$.
\end{rmk}	

The proof of Manin and Schechtman's theorem requires several pages. It is straightforward to show that $\leq_{\MS}$ is a partial order. That $B(n,k)$ is a ranked poset comes from Lemma \ref{DGH:lem:AdmFlip}. It is not obvious that $B(n,k)$ should have a unique minimal and maximal element. Nor is it evident that maximal chains in $B(n,k)$ should generate all the admissible $k+1$-orders on $C(n,k+1)$. Manin and Schechtman's proof is largely inductive in nature, working with representative orders which they deem to be ``good orders.'' 

Using these ``good orders'' Manin and Schechtman were able to prove that $\Inv(r) = \emptyset$ if and only if $r = r_{\min}$. They also showed that $N(r) \cap \Inv(r)$ is non-empty whenever $r \neq r_{\min}$. So there is always a sequence of packet flips, each one reducing the size of $\Inv(\rho)$, which connects $r_{\min}$ to $r$. Moreover, the involution $r \mapsto r^t$ reverses $\leq_{\MS}$. So the relation $r_{\min} \leq_{\MS} r^t$ implies that $r \leq_{\MS} r_{\max}$.

This approach is quite different from the one we take in this paper. That is because working directly with the Manin-Schechtman order is often quite difficult. Proving, via a sequence of packet flips, that two arbitrary admissible $k$-orders are related under the Manin-Schechtman partial order is a fairly abstruse proposition. Manin and Schechtman's ``good orders'' are only compatible when working with orders on $C(n,k)$. We do not know of an analog for orders on an arbitrary realizable $k$-set. Instead, it is more convenient to work with an alternate formulation provided by Ziegler.

\subsection{Realizable sets and the single step inclusion order}

Lemma \ref{DGH:lem:AdmFlip} along with Part \ref{MSthm:partfour} of Theorem $\ref{DGH:thm:MSthm}$ suggest that it may be easier to work directly with inversion sets rather than with equivalence classes of admissible $k$-orders. The question then becomes - is there a characterization of inversion sets which is amenable, under the correspondence in Theorem $\ref{DGH:thm:MSthm}.\ref{MSthm:partfour}$, to a partial ordering that is equivalent to the Manin-Schetchman partial order on $B(n,k)$? Ziegler affirmatively answered this question by characterizing inversion sets as \emph{realizable} $\boldsymbol{k+1}$\emph{-sets} (alt. consistent sets). See \cite{DGH:Z1993} for the original paper.

\begin{notation}\label{DGH:not:PreSuf} Let $T = \{t_1 < \dots < t_m\}$ be any totally ordered set. A \emph{prefix} of $T$ is any subset of the form $\{t_1,\ldots,t_i\}$. A prefix is \emph{proper} if $i < m$. A \emph{suffix} of $T$ is a subset of the form $\{t_i,\ldots, t_m\}$. A suffix is proper if $1 < i$.
\end{notation}

\begin{rmk}
	We will also refer to the prefix of a total order $\rho \in A(n,k)$. In this case we are referring to a prefix of $C(n,k)$ under the order $\rho$. The same is true for a suffix of a total order.
\end{rmk}

For any $(k+1)$-set $X = \{x_1,\ldots,x_{k+1}\}$, the packet of $X$ with its lexicographic order is $P_X = \{\widehat{x}_{k+1} < \ldots < \widehat{x}_{1}\}$. A prefix of some $k$-packet looks like $\{\widehat{x}_{k+1},\ldots,\widehat{x}_i\}$. A suffix looks like $\{\widehat{x}_{i},\ldots,\widehat{x}_{1}\}$.

\begin{example}\label{DGH:ex:PreSuf} The set $\{12, 13\}$ is a proper prefix of of the $2$-packet $P_X = \{12 < 13 < 23\}$. The set $\{13, 23\}$ is a proper suffix.
\end{example}



\begin{definition}\label{DGH:def:JsJpJFJ0} For any $J \subseteq C(n,k+1)$ define the following $(k+2)$-sets as:
	\begin{itemize}
		
		\item  $J_p \coloneqq \{X \in C(n,k+2) \mid P_X \cap J \text{ is a proper prefix of } P_X\}$
		
		\item  $J_s \coloneqq \{X \in C(n,k+2) \mid P_X \cap J \text{ is a proper suffix of } P_X\}$
		
		\item $J_F \coloneqq \{X \in C(n,k+2) \mid P_X \cap J = P_X\}$.
		
		\item $J_\emptyset \coloneqq \{x \in C(n,k+2) \mid P_X \cap J = \emptyset\}$.
	
	\end{itemize}
\end{definition}

\begin{notation}\label{DGH:not:Full} Whenever $X \in J_F$, we say that $X$ is \emph{full} in $J$. If $X \in J_{\emptyset}$, then $X$ is \emph{empty} in $J$.
\end{notation}

\begin{definition}
	$J$ is called a \emph{realizable} $\boldsymbol{k+1}$\emph{-set} if $P_X \cap J$ is either a prefix or a suffix for all $X \in C(n,k+1)$. Equivalently, $J$ is realizable if the union of $J_s$, $J_p$, $J_F$, and $J_\emptyset$ is all of $C(n,k+2)$. 
\end{definition}

\begin{rmk}\label{DGH:rmk:comp}
	If $J^c \coloneqq C(n,k)\setminus J$, then $J^c$ is also realizable as $(J^c)_s = J_p, (J^c)_s = J_s, (J^c)_F = J_{\emptyset}, (J^c)_{\emptyset} = J_F$.
\end{rmk}

In order to gain some familiarity with realizablity as well as the sets $J_s$, $J_p$, $J_F$, and $J_\emptyset$ it is worth briefly working through why this next lemma is true. We will often use it when $A = J_s$ and $B \subseteq J_s \cup J_F$.
\begin{lemma}\label{DGH:lem:ABpart}
	Let $A$ and $B$ be realizable $k$-sets such that $A \subset B$. Then we have:
	\begin{equation}
	\small{\begin{array}{ccc}
		B_{\emptyset} = B_{\emptyset} \cap A_{\emptyset}, & B_s = (B_s \cap A_s) \cup (B_s \cap A_{\emptyset}), & B_{p} = (B_p \cap A_p) \cup (B_p \cap A_{\emptyset})\\
		\multicolumn{3}{c}{B_{F} = (B_F \cap A_p) \cup (B_F \cap A_s) \cup (B_F \cap A_{\emptyset}) \cup (B_F \cap A_F).}
	\end{array}}
	\end{equation}

	\begin{equation}
		\small{\begin{array}{ccc}
			A_{\emptyset} = A_{\emptyset} \cap B_{\emptyset}, & A_s = (A_s \cap B_s) \cup (A_s \cap B_F), & A_p = (A_p \cap B_p) \cup (A_p \cap B_F)\\
			\multicolumn{3}{c}{A_F = A_F \cap B_F.}
		\end{array}}
	\end{equation}
\end{lemma}
\begin{proof}
	Left as an exercise.
\end{proof}

The definition of realizability presented in this paper is a reformulation of Ziegler's notion of consistency. Oftentimes it is easier to show that a $k$-set and its complement are convex, another equivalent reformulation of realizability. The following lemma establishes this equivalence. It also proves that every inversion set of an admissible $k$-order is a realizable $(k+1)$-set.

\begin{lemma}[Ziegler 2.4]\label{DGH:lem:Zreal}
	Every inversion set $U \subset C(n,k+1)$ satisfies the following equivalent conditions:
	\begin{enumerate}
		
		\item $U$ and its complement are both \emph{convex}: if $\{j_1 < j_2 < j_3\} \subseteq K$ for some $K \in C(n,k+2)$, then the intersection of $U$ with $\{\widehat{j}_3 < \widehat{j}_2 < \widehat{j}_1\}$ is neither $\{\widehat{j}_3, \widehat{j}_1\}$ nor $\{\widehat{j}_2\}$,
		
		\item $U$ is realizable, that is, the intersection of $U$ with any $(k+1)$-packet is a prefix or suffix of the packet.  
	
	\end{enumerate}
\end{lemma}

\begin{example}\label{DGH:ex:InvReal} As a special case of Lemma \ref{DGH:lem:Zreal}, the inversion sets of reduced expressions for $w_0 \in S_4$ are all realizable $3$-sets. There is a single $3$-packet when $n = 4$, $P_{1234}$. It contains $P_{1234} = \{123, 124, 134, 234\}$. Note that all the inversion sets in Example \ref{DGH:ex:S4Cor} are suffixes of $P_{1234}$. The inversion sets of entries on the other side of the rex graph are all prefixes. Moreover, the inversion set for each reduced expression in the reduced expression graph of $S_4$ (see Figure \ref{DGH:fig:S4rex}) is a prefix or suffix of $P_{1234}$.
\end{example}

Not only did Ziegler prove that every inversion set is realizable, but they also proved that every realizable $(k+1)$-set arises as the inversion set of some admissible $k$-order. Then, using the correspondence $r \leftrightarrow \Inv(r)$, Ziegler showed that the single step inclusion order on realizable $(k+1)$-sets is equivalent to the Manin-Schechtman order on $B(n,k)$. 

\begin{definition}\label{DGH:def:SSOrd}
	For any collection of sets $\mathfrak{U}$, we define the covering relationship
	\begin{equation}\label{SSOrd:Cover}
		U_1 \sqsubset U_2 \Leftrightarrow U_1 \subset U_2 \text{ and } \vert U_2\vert = \vert U_1\vert + 1
	\end{equation}
	The \emph{single step inclusion order on } $\boldsymbol{\mathfrak{U}}$ is then given by the relationship
	\begin{equation}\label{SSOrd:Ord}
		U \leq_{\SSO} \tilde{U} \Leftrightarrow \exists U_i's \in \mathfrak{U} \text{ with } U \sqsubset U_1 \sqsubset \dots \sqsubset \tilde{U} \text{ where } \vert U_i \setminus U_{i-1}\vert = 1.
	\end{equation}
\end{definition}

\begin{theorem}[Ziegler 4.1B] \label{DGH:thm:Zieg} Let $1 \leq k \leq n$. There is a natural isomorphism of posets between 
	\begin{enumerate}
		
		\item the higher Bruhat order $B(n,k)$,
		
		\item the set of all realizable subsets of $C(n,k+1)$, ordered by single step inclusion.
	
	\end{enumerate}
	The isomorphism is given by the map $r \mapsto \Inv(r)$ for any $r \in B(n,k)$. The unique minimal realizable $k$-set is $\emptyset$ and the unique maximal realizable $k$-set is $C(n,k)$.
\end{theorem}

\begin{rmk}\label{DGH:rmk:Ziegthm}
	A key takeawy of Ziegler's theorem is that if $\Inv(\rho)$ and $\Inv(\rho')$ differ by a single element $X$, then we can flip $P_X$ in both $\rho$ and $\rho'$.
\end{rmk}

\subsection{Admissibility and realizability}

An important consequence of Theorem \ref{DGH:thm:MSthm} and Theorem \ref{DGH:thm:Zieg} is the tight connection between the admissibility and realizability. In addition to the correspondence between realizable $(k+1)$-sets and admissible $k$-orders, one can characterize admissible $k$-orders precisely as those with realizable prefixes. 

\begin{lemma}\label{DGH:lem:AdmPre}
	When $k > 1$, a $k$-order $\rho$ is admissible exactly when every prefix of $\rho$ is a realizable $k$-set. 
\end{lemma}
\begin{proof}
	To begin suppose that $\rho$ is admissible. Let $J$ be a prefix of $\rho$. By definition, $\rho$ induces the lexicographic or antilexicographic order on every $k$-packet $P_X$. If $\rho$ induces the lexicographic order on $P_X$, then $J \cap P_X$ is a (not necessarily proper) prefix of $P_X$. If $\rho$ induces the antilexicographic order on $P_X$, the $J \cap P_X$ is a suffix of $P_X$. Hence $J$ is realizable. 
	
	Conversely suppose that every prefix of $\rho$ is a realizable $k$-set. There are only two orders on a packet $P_X$ all of whose prefixes are either prefixes or suffixes of the packet - the lexicographic and antilexicographic orders. Hence $\rho$ induces the lexicographic or antilexicographic order on every $k$-packet. It is therefore admissible. 
\end{proof}

The isomorphism of posets contained in Theorem $\ref{DGH:thm:Zieg}$ also gives additional necessary and sufficient conditions for when a set is realizable. In particular, a $k$-set is realizable if and only if it appears as the prefix of some admissible $k$-order (compare to Lemma \ref{DGH:lem:AdmPre}). These conditions can further be used to establish when two realizable $k$-sets are related under the Manin-Schechtman or single step inclusion order.

\begin{lemma}\label{DGH:lem:Zequiv}
	For $k > 1$ and $J \subseteq C(n,k)$ the following are equivalent:
	\begin{enumerate}
	
		\item $J$ is realizable,\label{Zequiv:partone}
	
		\item there exists a $\rho \in A(n,k)$ for which $J$ is a prefix, ie for all $x \in J$ and $y \in J^c$ we have $x \leq_{\rho} y$, \label{Zequiv:parttwo}
	
		\item there exists a $\rho \in A(n,k)$ for which $J_s \subseteq \Inv(\rho)$ and $J_p \subseteq \Inv(\rho)^c$. \label{Zequiv:partthree}
	
	\end{enumerate}
\end{lemma}
\begin{proof}
	\begin{description}
	
		\item[\ref{Zequiv:partone} $\Leftrightarrow$\ref{Zequiv:parttwo}] When $k > 1$, Theorem \ref{DGH:thm:Zieg} shows that $J$ is realizable if and only if there exists $r \in B(n,k-1)$ with $\Inv(r) = J$. Pick a path in $B(n,k-1)$ from $\emptyset \rightarrow J \rightarrow C(n,k)$. Such a chain exists by Theorem \ref{DGH:thm:MSthm}.\ref{MSthm:parttwo}. This chain is maximal. By Theorem \ref{DGH:thm:Zieg} the maximal chains in $B(n,k-1)$ are in bijection with $A(n,k)$. So there exists $\rho \in A(n,k)$ which orders realizable $k$-sets according to the maximal chain we selected. This implies \ref{Zequiv:parttwo}. The reverse implication comes from Lemma \ref{DGH:lem:AdmPre}.
	
		\item[\ref{Zequiv:parttwo} $\Rightarrow$\ref{Zequiv:partthree}] Let $\rho \in A(n,k)$ satisfy the conditions in \ref{Zequiv:parttwo}. Take $U \in J_s$. Then there exists $x \in P_U \cap J$ and $y \in P_U \cap J^c$ for which $y \leq_{\lex} x$. Because $J$ is a prefix of $\rho$, we know that $x \leq_{\rho} y$. Hence $U \in \Inv(\rho)$. 
		
		Now let $U \in J_p$. This time there exists $x \in P_U \cap J$ and $y \in P_U \cap J^c$ for which $x \leq_{\lex} y$. Because $J$ is a prefix of $\rho$, we must have $x \leq_{\rho} y$. Hence $U \notin \Inv(\rho)$.
		
		\item[\ref{Zequiv:partthree} $\Rightarrow$\ref{Zequiv:parttwo}] Let $\rho \in A(n,k)$ satisfy the conditions in part \ref{Zequiv:partthree}. If $J$ is already a prefix of $\rho$ then we are done. If not, we claim that $\rho$ is equivalent to some admissible order with $J$ as a prefix. 
		
		Consider a pair $(x,y)$ where $x \in J$ and $y \notin J$. We claim that if $x$ and $y$ are in some shared packet $P_K$, then $x \leq_{\rho} y$. Because $J$ is realizable, there are only two cases to consider, $K \in J_s$ and $K \in J_p$. If $K \in J_s$, then $y \leq_{\lex} x$. Because $J_s \subseteq \Inv(\rho)$, it follows that $x \leq_{\rho} y$. If $K \in J_p$, then $x \leq_{\lex} y$, and $x \leq_{\lex} y$ because $J_p \subseteq \Inv(\rho)^c$.
		
		Let $\bar{J}$ be the minimal prefix of $\rho$ that contains $J$. Based on the preceding argument, the set $\{(x,y) \in \bar{J} \mid x \in J, y \notin J, y \leq_{\rho} x\}$ consists entirely of pairs without a shared packet. These are also precisely all the pairs keeping $J$ from being a prefix of $\rho$. At least one of these pairs is consecutive in $\rho$. Swapping that pair and inducting on the number of offending pairs proves the claim.
	\end{description}
\end{proof}

\begin{corollary}\label{DGH:cor:Zequiv} Let $k > 1$ and $J$ be a realizable $k$-set. If $\rho$ is an admissible $k$-order for which $J_s \subseteq \Inv(\rho)$ and $J_p \subseteq \Inv(\rho)^c$, then $\rho$ is elementary equivalent to an admissible $k$-order $\rho'$ for which $J$ is a prefix.
\end{corollary}
\begin{proof}
	This is the content of the last part of the proof of Lemma \ref{DGH:lem:Zequiv}.
\end{proof}

One way to look at Lemma \ref{DGH:lem:Zequiv} is that $J$ is realizable if and only if $\emptyset \leq_{\SSO} J$. Building off of this idea, we can come up with similar conditions for when two arbitrary realizable $k$-sets are related under the single-step inclusion order and, equivalently, the Manin-Schechtman order.

\begin{lemma}\label{DGH:lem:SSrel} Let $k > 1$ and $J \subset J'$ be realizable $k$-sets. The following are equivalent:
	\begin{enumerate}
	
		\item $J \leq_{\SSO} J'$, \label{SSrel:one}
	
		\item There exists some $\rho \in A(n,k)$ for which $J$ and $J'$ are both prefixes of $\rho$, with $J$ coming before $J'\setminus J$, \label{SSrel:two}
	
		\item There exists some $\rho \in A(n,k)$ for which $(J_s \cup J'_s) \subseteq \Inv(\rho)$ and $(J_p \cup J'_p) \subseteq \Inv(\rho)^c$. \label{SSrel:three}
	
	\end{enumerate}
\end{lemma}
\begin{proof}
	The equivalence of parts \ref{SSrel:one} and \ref{SSrel:two} is argued similarly to the equivalence of parts \ref{Zequiv:partone} and \ref{Zequiv:parttwo} from Lemma \ref{DGH:lem:Zequiv}. We will prove \ref{SSrel:two}$\Leftrightarrow$\ref{SSrel:three}.
	\begin{description}
	
		\item[\ref{SSrel:two}$\Rightarrow$\ref{SSrel:three}] This direction is immediate from Lemma \ref{DGH:lem:Zequiv}.
		
	
		\item[\ref{SSrel:three}$\Rightarrow$\ref{SSrel:two}] From Lemma \ref{DGH:lem:Zequiv} any $\rho \in A(n,k)$ satisfying part \ref{SSrel:three} is equivalent to one where $J'$ are prefixes. Without loss of generality, assume that this is true of $\rho$. We only need to prove that $\rho$ is equivalent to some order where $J$ comes before $J'\setminus J$. 
		
		As in the proof of Lemma \ref{DGH:lem:Zequiv}, consider a pair $(x,y)$ where $x \in J$ and $y \in J'\setminus J$. Again, we claim that if $x$ and $y$ are in a shared packet $P_K$, then $x \leq_{\rho} y$. The reasoning is as before. If $K \in J_s$, then $y \leq_{\lex} x$. Hence $x \leq_{\rho} y$ be virtue of $J_s$ begin a subset of $\Inv(\rho)$. If $K \in J_p$, then $x \leq_{\lex} y$, and $x \leq_{\rho} y$.
		
		The set $\{(x,y) \in J \times (J'\setminus J) \mid y \leq_{\rho} x\}$ then consists entirely of pairs which are not in a shared packet. We know that at least one such pair must be consecutive. So $\rho$ is elementary equivalent to another order with one fewer out-of-order pair. Induct on the number of out-of-order pairs to complete the proof.
	\end{description}
\end{proof}

\begin{notation}Throughout this paper we will often abuse terminology and refer to realizable $k+1$-sets as if they were elements of $B(n,k)$. For instance, given realizable $k+1$-sets $K$ and $L$ we might discuss ``a path from $K$ to $L$ in $B(n,k)$''. Here what we really mean is that $K = \Inv(r_1)$ and $L = \Inv(r_2)$ for elements $r_1,r_2 \in B(n,k)$ and that there is a chain from $r_1 \rightarrow r_2$ within $B(n,k)$. This abuse is justified by the isomorphism in Theorem $\ref{DGH:thm:Zieg}$
\end{notation}

\section{Admissible orders on realizable sets}\label{DGH:Sec:Adm}


When $k = 2$, Corollary \ref{DGH:cor:RexGraph} tells us that $B(n,2)$ corresponds to the reduced expression graph of $w_0$ in $S_n$. It is natural to wonder if we can replicate this for a non-longest word $w$ in $S_n$. Can we realize equivalence classes of reduced expressions for an arbitrary word in $S_n$ as some partially ordered set? As we saw in subsection \ref{DGH:MSthm:MSThm}, reduced expressions for $w_0$ corresponded to admissible $2$-orders on $\Inv(w_0)$. Meanwhile $\Inv(w)$ is a realizable $2$-set. Is there some generalization of admissibility for orders on an arbitrary realizable $k$-set $J$? Moreover, is there some sort of analog to Theorem \ref{DGH:thm:MSthm} which allows us to inductively create these partially ordered set? The rest of this paper is dedicated to providing an affirmative answer to this question when $k = 2$. For the remainder of this paper, assume that $J$ is a realizable $k$-set where $k \geq 2$.

\subsection{Paths from $\emptyset$ to $J$ in $B(n,k)$} \label{DGH:Two:One}

Later on we will define the i-th Bruhat order of a realizable $2$-set as a poset of equivalence classes of certain paths in $B(n,i-1)$. Before doing so, we need to establish notation and extend many of the definitions in Section \ref{DGH:Sec:One}. 

\begin{notation}
	The set of all paths from $\emptyset$ to $J$ in $B(n,k)$ is denoted $\paths{\emptyset}{J}{k}$.
\end{notation}

Note that any path $\gamma \in \paths{\emptyset}{J}{k}$ gives a total order on $J$. In Section \ref{DGH:Two:Two} we precisely characterize these total orders. For now we proceed by extending all of our Section \ref{DGH:Sec:One} definitions to these total orders.

\begin{definition}\label{DGH:def:InvSetsPath}
	Take $\gamma \in \paths{\emptyset}{J}{k}$. Define the \emph{inversion set of $\gamma$} as 
	\begin{equation}
	\Inv(\gamma) = \{X \in J_F \mid \gamma\vert_{P_X} = \rho_{\anti}\vert_{P_X}\} \cup J_s.
	\end{equation}
\end{definition}

\begin{rmk}\label{DGH:rmk:InvTop}
	We define the inversion set of any path $\tau$ from $J \rightarrow C(n,k)$ as $\Inv(\tau) \coloneqq \Inv(\tau^t)^c$. This exploits the fact that the transpose of a suffix of any maximal chain in $B(n,k-1)$ is the prefix of some other maximal chain. 
\end{rmk}

Note that $J_s \subseteq \Inv(\gamma)$ and $J_p \cap \Inv(\gamma) = \emptyset$ for any $\gamma \in \paths{\emptyset}{J}{k-1}$. This is similar to the condition on $\rho$ from Lemma \ref{DGH:lem:Zequiv}. However, here $\gamma$ induces an order on $J$ whereas earlier $\rho$ was an order on $C(n,k)$. 

When $J = C(n,k+1)$, the inversion set of an admissible $(k+1)$-order is a special case of Definition \ref{DGH:def:InvSetsPath}. Then paths from $\emptyset$ to $J$ in $B(n,k)$ are just the maximal chains in $B(n,k)$. By Theorem \ref{DGH:thm:MSthm}(\ref{MSthm:partfour}), the collection of maximal chains in $B(n,k)$ corresponds to $A(n,k+1)$. The inversion set of a maximal chain is precisely the inversion set of the corresponding admissible $(k+1)$-order.

\begin{rmk}\label{DGH:rmk:Done}
	Definition \ref{DGH:def:InvSetsPath} roughly states that $\Inv(\gamma)$ is the set of $X$ for which $\gamma$ is antilex on $P_X \cap J$. We will see in Lemma \ref{DGH:lem:0JPaths} that $\gamma$ must induce the antilexicographic order on $P_X \cap J$ for any $X \in J_s$. The above definition is slightly more precise when dealing with the case when $P_X \cap J$ is a singleton (which may be true for $X \in J_s$ or $X \in J_p$).
\end{rmk}

\begin{example}\label{DGH:ex:InvPaths} 	Here is the inversion set of a path in $B(4,2)$ (see Figure \ref{DGH:fig:S4rex}). The path is written down as an ordering of $J$.
	\begin{center}
		\begin{tabular}{c|c|c|c} 
			\multicolumn{1}{c}{Graph}& \multicolumn{1}{c}{$J$} & \multicolumn{1}{c}{Path}&\multicolumn{1}{c}{Inversion Set}\\
			\hline $B(4,2)$ & $\{134,234,124\}$& $234 < 134 < 124$ & $\{1234\}$
		\end{tabular}
	\end{center}  
Notice that the inversion set of the path is realizable. We will eventually prove Theorem \ref{DGH:thm:InvReal} that the inversion set of any path from $\emptyset$ to $J$ is realizable.
\end{example}

Analogous to the function $N: A(n,k) \rightarrow 2^{C(n,k+1)}$ from earlier, we define the function $N: \paths{\emptyset}{J}{k} \rightarrow 2^{J_F}$ by
\begin{equation}
	N(\gamma) = \{X \in J_F \mid P_X \text{ forms a chain in } \gamma\}.
\end{equation}
Take note that $N(\gamma)$ is a subset of $J_F$. It is impossible for a whole packet to make a chain in $\gamma$ if it is not full in $J$. Again, when $J = C(n,k+1)$ we recover the old definition of $N(\cdot)$ on $A(n,k+1)$. Packet flips are defined analogously as well. Again, they either increase or decrease the size of an inversion set by one element.

\begin{definition}\label{DGH:def:FlipPaths}
	Let $\gamma \in \paths{\emptyset}{J}{k}$ and $X \in N(\gamma)$. Define $p_X(\gamma)$ as the total order on $J$ in which the chain $P_X$ is reversed while all the other elements remain the same.
\end{definition}

We have defined $p_X(\gamma)$ as a total order on $J$. We should not expect all total orders of $J$ to appear as paths between $\emptyset$ and $J$ in $B(n,k)$. Fortunately, $p_X(\gamma)$ does.

\begin{lemma}\label{DGH:lem:PathFlips}
	If $\gamma \in \paths{\emptyset}{J}{k}$ and $X \in N(\gamma)$, then $p_X(\gamma) \in \paths{\emptyset}{J}{k}$.
\end{lemma}
\begin{proof}
	Extend $\gamma$ to a path from $\emptyset \stackrel{\gamma}{\rightarrow} J \rightarrow C(n,k+1)$ in $B(n,k)$. This path corresponds to some admissible $(k+1)$-order $\rho$. Notice that $X \in N(\rho)$. As $X \in J_F$, $p_X(\rho)$, is still a path from $\emptyset \rightarrow J \rightarrow C(n,k+1)$. The subpath taken by $p_X(\rho)$ from $\emptyset \rightarrow J$ induces the same total order on $J$ as $p_X(\gamma)$. So $p_X(\gamma)$ is realized as a path between $\emptyset$ and $J$.
\end{proof}

\begin{definition}\label{DGH:def:EquivPaths} Two paths $\gamma, \gamma' \in \paths{\emptyset}{J}{k}$ are \emph{elementary equivalent} if they differ by the interchange of two neighboring elements that are not in a shared $(k+1)$-packet. 
	
Define $\sim$ to be the equivalence relation generated by these elementary equivalences. We define
	\begin{equation}\label{DGH:eqn:Bpaths}
		\epaths{\emptyset}{J}{k} = \paths{\emptyset}{J}{k} \Big/ \sim
	\end{equation}
	Let $\pi: \paths{\emptyset}{J}{k} \rightarrow \epaths{\emptyset}{J}{k}$ be the quotient map.
\end{definition}

Like before, elementary equivalent paths have the same inversion sets as each other. The function $\Inv(\cdot)$ descends to a well-defined function on $\epaths{\emptyset}{J}{k}$. Packet flips and $N(\cdot)$ also descend to $\epaths{\emptyset}{J}{k}$ in the same way that their analogs on $A(n,k+1)$ descended to $B(n,k+1)$. The proofs go through mutatis mutandis. We can also define an order on $\epaths{\emptyset}{J}{k}$ which reproduces the Manin-Schechtman order when $J_1 = \emptyset$ and $J_2 = C(n,k+1)$.

\begin{definition}\label{DGH:def:MSOrdPaths}
	Define the following relation on $\epaths{\emptyset}{J}{k}$:
	\begin{equation*}\label{DGH:rel:MSrel}
		\begin{aligned}
			r \leq_{\MS} r' \Leftrightarrow &\text{ there exists a sequence of } X_i \in J_F \text{ such that }\\			
			&\flipup{\flipup{r=r_1}{r_2}{X_1}}{}{X_2} \ldots \flipup{}{\flipup{r_i}{r_{i+1} = r'}{X_i}}{X_{i-1}}.
		\end{aligned}
	\end{equation*}
\end{definition}

\subsection{Admissible $J$-orders} \label{DGH:Two:Two}

Paths $\gamma \in \paths{\emptyset}{J}{k-1}$, in a very straightforward manner, generalize admissible $k$-orders to arbitrary realizable $k$-sets. In this subsection, we show that they correspond to admissible $J$-orders. We will also show that they have realizable inversion sets. The posets $\epaths{\emptyset}{J}{k-1}$ will also have unique minimal and maximal elements under $\leq_{\MS}$. This will lead to the definition of the second Bruhat order for $J$. First we show that any $\gamma \in \paths{\emptyset}{J}{k-1}$ induces a particularly nice order on $J$.

\begin{lemma}\label{DGH:lem:0JPaths}
	Let $\gamma \in \paths{\emptyset}{J}{k-1}$. As a total order on $J$, $\gamma$ induces
	\begin{itemize}
		
		\item ...the antilexicographic ordering on $P_X \cap J$ when $X \in J_s$,
		
		\item ...the lexicographic ordering on $P_X \cap J$ when $X \in J_p$,
		
		\item ...either the lexicographic or antilexicographic ordering on $P_X$ when $X \in J_F$.
		
	\end{itemize}
\end{lemma}
\begin{proof}
	Any path $\gamma \in \paths{\emptyset}{J}{k-1}$ can be extended to a maximal chain $\rho$ in $B(n,k-1)$ with $\gamma = \rho\vert_{J}$. By the Manin-Schechtman correspondence, $\rho$ is an admissible $(k+1)$-order. Thus it induces either the lexicographic or antilexicographic order on every $k$-packet. So $\gamma$ must induce either the lexicographic or antilexicographic order on $P_X \cap J$ for every $k$-packet.
	
	Because $J$ is a prefix of $\rho$, Lemma \ref{DGH:lem:Zequiv} shows that $J_s \subseteq \Inv(\rho)$ and $J_p \subseteq \Inv(\rho)^c$. Thus $\gamma$ induces the lexicographic order on $P_X \cap J$ for ever $X \in J_p$ and the antilexicographic order on $P_X \cap J$ for ever $X \in J_s$.
\end{proof}

\begin{definition}\label{DGH:def:admJ}
An order $\rho$ on $J$ is \emph{admissible} (alt. is an admissible $J$-order) if $\rho$ satisfies the following for any $k$-packet $P_X$:
\begin{itemize}
	
	\item $\rho$ induces the antilexicographic ordering on $P_X$ when $X \in J_s$.
	
	\item $\rho$ induces the lexicographic ordering on $P_X$ when $X \in J_p$.
	
	\item $\rho$ induces either the lexicographic or antilexicographic ordering on $P_X$ when $X \in J_F$. 
	
\end{itemize}

Let $A_J(n,k)$ denote the set of admissible $k$-orders on $J$.
\end{definition}

So far we have shown that $\paths{\emptyset}{J}{k} \hookrightarrow A_J(n,k)$. This embedding is actually a bijection. To avoid needless repetition, note that the definitions of inversion sets, packet flips, etcetera for $A_J(n,k)$ are the obvious adaptations of the same notions for $A(n,k)$ and $\paths{\emptyset}{J}{k}$. The definitions will intertwine with the embedding $\paths{\emptyset}{J}{k} \hookrightarrow A_J(n,k)$. 

\begin{rmk} The inversion set of an admissible $J$-order $\rho$ contains $J_s$ and is a subset of $J_s \cup J_F$.
\end{rmk}


In order to prove the desired bijection, we first must show that the inversion set of every admissible $J$-order is realizable. The proof requires a few auxiliary lemmas in order to address subtleties not present for admissible $k$-orders. When $J$ is a strict subset of $C(n,k)$, $J_{\emptyset}$ may be non-empty and one might worry that it causes $\Inv(\rho)$ to be non-realizable. Ultimately this does not happen. This lemma generalizes Ziegler's Lemma 2.6 which, in our notation, states that $J_s \cup J_F \cup J_{\emptyset}$ (alt. $J_p \cup J_F \cup J_{\emptyset}$ )is a realizable $(k+1)$-set \cite[\S2]{DGH:Z1993}.

\begin{lemma}\label{DGH:lem:Jreal} Whenever $J$ is a realizable $k$-set the following are true:
	\begin{enumerate}
		
		\item $J_s$ is a realizable $(k+1)$-set (alt. $J_p$)\label{Jreal:partone},
		
		\item No $(k+1)$-packet intersects both $J_F$ and $J_{\emptyset}$ non-trivially. \label{Jreal:parttwo}
	\end{enumerate}
\end{lemma}

Part 1 follows immediately from Ziegler's Lemma 2.6, the realizability of $J$, and the fact that the  complement of a realizable set is itself realizable. However, it is worth proving part 1 directly to illustrate how proofs of realizability work in general. We do this by proving that $J_s$ and its complement are convex. In this way, our proof differs from Ziegler's. 
	
\begin{proof}	
	Let $\{a < b < c\} \subseteq X$ for some $X \in C(n,k+2)$. We need to show that the intersection of $J_s$ with $\{\widehat{c},\widehat{b},\widehat{a}\}$ is neither $\{\widehat{c},\widehat{a}\}$ nor $\{\widehat{b}\}$. We now proceed via cases.
	
	\begin{description}
		
		\item[$\widehat{b}\notin J_s$, $\widehat{b} \in J_F$] In this case $\widehat{b}\widehat{k+2} \in J$. This element is contained inside of $P_{\widehat{k+2}}$. It also falls in between $\widehat{c}\widehat{k+2}$ and $\widehat{a}\widehat{k+2}$ in the lexicographic order. As the realizability of $J$ implies that $P_{\widehat{k+2}} \cap J$ is either a suffix or prefix of $P_{\widehat{k+2}}$, either $\widehat{a}\widehat{k+2}$ or $\widehat{c}\widehat{k+2}$ is in $J$ (or $\widehat{k+1}\widehat{k+2} \in J$ if $c = k+2$). These represent the lexicographic minimal elements in $P_{\widehat{a}}$ and $P_{\widehat{c}}$ respectively. Hence, one of $\widehat{a}$ and $\widehat{c}$ is not in $J_s$. 
		
		\item[$\widehat{b}\notin J_s, \widehat{b}\notin J_F$] Based on the realizability of $J$, $\widehat{b} \in J_p \cup J_{\emptyset}$. In either case, $\widehat{1}\widehat{b} \notin J$. Again, the realizability of $J$ applied to $P_{\widehat{1}}$ dictates that one of $\widehat{1}\widehat{c}$ or $\widehat{1}\widehat{a}$ (or $\widehat{1}\widehat{2}$ if $a = 1$) not be in $J$. In either case, one of $\widehat{a}$ or $\widehat{c}$ is not in $J_s$. 
	\end{description}
	We may now assume that $\widehat{b} \in J_s$. This implies that $\widehat{b}\widehat{k+2} \notin J$, while $\widehat{b}\widehat{1} \in J$. Our goal is to show that one or both of $\widehat{a}$ and $\widehat{c}$ are also in $J_s$. We assume to the contrary and proceed via cases. 
	\begin{description}
		
		\item[$\widehat{a},\widehat{c}\in J_F$] By assumption, $\widehat{a}\widehat{k+2}, \widehat{c}\widehat{k+2} \in J$ (or $\widehat{k+1}\widehat{k+2}$ if $c = k+2$). As $\widehat{b}\widehat{k+2}$ falls between $\widehat{c}\widehat{k+2}$ and $\widehat{a}\widehat{k+2}$ in the lexicographic order, the realizability of $J$ implies that $\widehat{b}\widehat{k+2} \in J$ as well. However, this implies that $\widehat{b} \in J_p \cup J_F$, which is a contradiction.

		\item[$\widehat{a},\widehat{c}\notin J_F$] By assumption, $\widehat{a},\widehat{c} \in J_p \cup J_{\emptyset}$. Necessarily, neither $\widehat{1}\widehat{a}$ (or $\widehat{1}\widehat{2}$ if $a = 1$) nor $\widehat{1}\widehat{c}$ are in $J$. However, $\widehat{1}\widehat{b} \in J$. So the intersection of $J$ with $\widehat{1}$ is neither a prefix nor a suffix. This contradicts the realizability of $J$. 
		
		\item[$\widehat{a} \in J_F, \widehat{c}\notin J_F$] We will show that the intersection of $J$ with $\widehat{k+2}$ is neither a prefix nor a suffix. By assumption, $\widehat{c} \in J_p \cup J_{\emptyset}$. Due to the fullness of $\widehat{a}$, $\widehat{a}\widehat{c} \in J \cap \widehat{c}$. Hence $\widehat{c} \in J_p$. Thus $\widehat{c}\widehat{k+2} \in J$ (or $\widehat{k+1}\widehat{k+2}$ if $c = k+2$) and $\widehat{a}\widehat{k+2} \in J$ since $\widehat{a}\in J_F$. However, by assumption, $\widehat{b}\widehat{k+2} \notin J$, which gives the desired contradiction.
		
		\item[$\widehat{c} \in J_F, \widehat{a}\notin J_F$] This plays out similarly to the previous case, but uses $\widehat{1}$ in place of $\widehat{k+2}$.
	
	\end{description}	
	
	This completes the proof of part 1. 
	
	To prove part 2 assume that $\widehat{a} \in P_X \cap J_F$. For any $d \neq a$, we get that $\widehat{a}\widehat{d} \in J$. So $\widehat{d} \notin J_{\emptyset}$ for all $d$.

\end{proof}

\begin{corollary}\label{DGH:cor:Jpreal} If $J$ is a realizable $k$-set, then $J_p$ is a realizable $(k+1)$-set.
\end{corollary}

\begin{proof}
	Recall that $J_p = (J^c)_s$ (Remark \ref{DGH:rmk:comp}). 
\end{proof}

\begin{lemma}\label{DGH:lem:parts}
	If $J$ is a realizable $k$-set, then the constituent pieces $P_X \cap J_s$, $P_X \cap J_p$, $P_X \cap J_F$, and $P_X \cap J_{\emptyset}$ of any $k+1$-packet can be arranged in one of four ways
	\begin{align}
			\{P_X \cap J_s < P_X \cap J_F < P_X \cap J_p\}\label{parts:sFp}\\
			\{P_X \cap J_p < P_X \cap J_F < P_X \cap J_s\}\label{parts:pFs} \\
			\{P_X \cap J_s < P_X \cap J_{\emptyset} < P_X \cap J_p\}\label{parts:s0p} \\
			\{P_X \cap J_p < P_X \cap J_{\emptyset} < P_X \cap J_s\}\label{parts:p0s}
	\end{align}
	Here the notation $P_X \cap J_s < P_X \cap J_F$ signifies that for every pair $A \in P_X \cap J_s$ and $B \in P_X \cap J_F$, we have $A \leq_{\lex} B$. It is possible that the constituent parts of each arrangement may be empty.
\end{lemma}
\begin{proof}
	Let $P_X$ be a $(k+1)$-packet. Realizability of $J$ is equivalent to $J_s \cup J_p \cup J_F \cup J_{\emptyset} = C(n,k+1)$. Thus $P_X = (P_X \cap J_s) \cup (P_X \cap J_p) \cup (P_X \cap J_F) \cup (P_X \cap J_{\emptyset})$.
	
	Because both $J_s$ and $J_p$ are realizable, their intersections with $P_X$ must be a prefix or suffix. If they both intersect non-trivially with $P_X$, then their intersections with $P_X$ must appear on opposite ends of $P_X$. This shows that any $(k+1)$-packet can be arranged as either
	\begin{align}
		\begin{split}
			\{P_X \cap J_s < &\text{ ??} < P_X \cap J_p\} \\
			&\text{or} \\
			\{P_X \cap J_p < &\text{ ??} < P_X \cap J_s\}.
		\end{split}
	\end{align}
	
	According to Lemma \ref{DGH:lem:Jreal} Part \ref{Jreal:parttwo}, the middle ?? segment is either $P_X \cap J_{\emptyset}$ or $P_{X} \cap J_F$. It cannot contain entries from both sets.
\end{proof}

\begin{definition}\label{DGH:def:Seg}
	When we divide a packet $P_X$ into lexicographic intervals based on its intersections with $J_s$, $J_p$, $J_F$, and $J_{\emptyset}$ (as in Lemma \ref{DGH:lem:parts}), we call this a \emph{segmentation} of $P_X$. With the exception of Lemma \ref{DGH:lem:parts}, we typically assume that each interval is non-empty. 
\end{definition}

Lemma \ref{DGH:lem:parts} is particularly useful because it automatically proves that a bevy of combinations of $J_s, J_p, J_{\emptyset}, \text{ and } J_F$ are realizable $(k+1)$-sets. For instance, $J_s \cup J_F$ is realizable. So is $J_s \cup J_F \cup J_{\emptyset}$. However, a set such as $J_s \cup J_F \cup J_p$ may or may not be realizable; it depends upon the existence of a packets with segmentation \eqref{parts:s0p} or and \eqref{parts:p0s}. We include following corollary now because it is an immediate consequence of Lemma \ref{DGH:lem:parts}. However, it will not be used until Section \ref{DGH:ssec:GenRes}.

\begin{corollary}\label{DGH:cor:badsets} Let $J$ be a realizable $k$-set and fix $X \in C(n,k+2)$. Then the following segmentations of $P_X$ are forbidden:
	\begin{align}
		&\{\emptyset < P_X \cap J_F < P_X \cap J_s\}\label{badsets:Fs} \\
		&\{P_X \cap J_p < P_X \cap J_F < \emptyset\}\\
		&\{P_X \cap J_s < P_X \cap J_{\emptyset} < \emptyset\} \\
		&\{\emptyset < P_X \cap J_{\emptyset} < P_X \cap J_p\}
	\end{align}
	For example, the Equation \eqref{badsets:Fs} states that the segmentation in Equation \eqref{parts:pFs} is impossible when $P_X \cap J_p$ is empty, but the other two segments are non-empty. All foure of these forbidden segmentations are associated with those found in Equations \eqref{parts:pFs} and \eqref{parts:s0p}.
\end{corollary}
\begin{proof}
	We shall only prove that the first segmentation is impossible. The other cases are similar or come from replacing $J$ with $J^c$.
	
	If $\widehat{k+2} \in J_F$, then $\widehat{1}\widehat{k+2} \in J$. This is the lexicographical minimum element in $P_{\widehat{1}}$. Therefore $\widehat{1} \in J_F \cup J_p$. 
\end{proof}

\begin{theorem}\label{DGH:thm:InvReal}
	Let $\rho$ be an admissible $J$-order. Then $\Inv(\rho)$ is a realizable $(k+1)$-set.
\end{theorem}

\begin{proof}
	As in the proof of Lemma $\ref{DGH:lem:Jreal}.\ref{Jreal:partone}$ we will exploit Lemma $\ref{DGH:lem:SSrel}$ and instead prove that $\Inv(\rho)$ and its complement are convex. Suppose that $\{a < b < c\} \subset X$ for some $X \in C(n,k+2)$. 
	Let us recall Ziegler's proof when $J = C(n,k)$\cite[Lemma 2.4]{DGH:Z1993}. In that case, we have the following
	\begin{equation}\label{InvReal:Asmpt}
		\widehat{a}\widehat{b}, \widehat{a}\widehat{c},\widehat{b}\widehat{c} \in J.
	\end{equation} 
	If 
	\begin{equation}\label{InvReal:bout}\tag{Case 1}
		\begin{array}{ccc}
			\widehat{c} \in \Inv(\rho),& \widehat{b} \notin \Inv(\rho) ,& \widehat{a} \in \Inv(\rho)
	\end{array}\end{equation}
	then we may conclude that 
	\begin{equation}\label{InvRel:Con}
		\begin{split}
			\widehat{a}\widehat{c} <_{\rho} \widehat{b}\widehat{c},\hspace{.1cm} &\hspace{.25cm} \widehat{a}\widehat{b} >_{\rho} \widehat{b}\widehat{c},\\
			\widehat{a}\widehat{b} &<_{\rho} \widehat{a}\widehat{c}
		\end{split}
	\end{equation}
	which is a contradiction. A similar contradiction occurs if 
	\begin{equation}\label{InvReal:bin}\tag{Case 2}
	\widehat{a},\widehat{c} \notin \Inv(\rho), \widehat{b} \in \Inv(\rho).
	\end{equation}
	
	Unfortunately, when $J \neq C(n,k)$, it is no longer clear that \eqref{InvReal:Asmpt} holds. Below we show that it does in fact hold in both cases, unless a more immediate contradiction is reached.
	\begin{description}
		
		\item[Case 1] In this case both $\widehat{a},\widehat{c} \in J_s \cup J_F$. By realizability, $\widehat{b} \in J_s \cup J_F$ as well. As $J_s \subseteq \Inv(\rho)$ and $\widehat{b} \notin \Inv(\rho)$, we know that $\widehat{b} \in J_F$. Because $\widehat{c} \leq_{\lex} \widehat{b} \leq_{\lex} \widehat{a}$, Lemma \ref{DGH:lem:parts} implies that one or both of $\widehat{a}$ and $\widehat{c}$ are in $J_F$. Suppose $\widehat{a} \in J_F$. Then 
		\begin{equation*}
			\begin{array}{ccc}
				\widehat{a}\sdiff{c} = \widehat{a}\widehat{c}, & \widehat{b}\sdiff{a} = \widehat{a}\widehat{b}, & \widehat{b}\sdiff{c} = \widehat{b}\widehat{c}
			\end{array}
		\end{equation*}
		are all in $J$. The same is true when $\widehat{c}\in J_F$. 
		
		\item[Case 2] In this case, $\widehat{b} \in J_s \cup J_F$ and $\widehat{a},\widehat{c} \notin J_s$. In fact, $\widehat{b} \in J_F$. Otherwise, the realizability of $J_s$ would imply that one or both of $\widehat{a}$ and $\widehat{c}$ was in $J_s$. Because $\widehat{b}\in J_F$ we know that $\widehat{a}\widehat{b}, \widehat{b}\widehat{c} \in J$.
		
		Moreover because $\widehat{b} \in J_F$, we can again use Lemma \ref{DGH:lem:parts} to conclude that at least one of $\widehat{a}$ and $\widehat{c}$ is in $J_F$. Either way, it follows that $\widehat{a}\widehat{c} \in J$, and \eqref{InvReal:Asmpt} holds. 	
	\end{description}
\end{proof}

\begin{theorem}\label{DGH:thm:AdmPathCorr}
	When $k > 1$, there is a bijection between $A_J(n,k)$ and $\paths{\emptyset}{J}{k-1}$. 
\end{theorem}
\begin{proof}
	Let $\rho \in A_J(n,k)$ be an admissible order on $J$. By Theorem \ref{DGH:thm:InvReal}, $\Inv(\rho)$ is a realizable $(k+1)$-set. By Theorem \ref{DGH:thm:Zieg}, there exists $\tilde{\rho} \in A(n,k)$ with $\Inv(\tilde{\rho}) = \Inv(\rho)$. By design $J_s \subseteq \Inv(\rho)$ and $J_p \subset \Inv(\rho)^c$. Therefore $J_s \subseteq \Inv(\tilde{\rho})$ and $J_p \subseteq \Inv(\tilde{\rho})$. By Lemma \ref{DGH:lem:Zequiv}, altering $\tilde{\rho}$ by a sequence of elementary equivalences, we can assume that $J$ is a prefix of $\tilde{\rho}$.
	
	We claim that, up to a sequence of elementary equivalences, $\tilde{\rho}\vert_{J} = \rho$. The argument is similar to that in Lemma \ref{DGH:lem:Zequiv}. Look at the set of out-of-order pairs $Y \coloneqq \{(x,y) \in J \mid \tilde{\rho} \text{ and } \rho \text{ induce the opposite order on } x \text{ and } y\}$. Since $\Inv(\tilde{\rho}) = \Inv(\rho)$, any such pair $(x,y)$ cannot be in a shared packet. Furthermore, at least one such pair must be consecutive under the ordering of $\rho$. Otherwise, transitivity would imply that the orders agree. So we can swap this pair and then use induction on the size of $Y$ to prove the claim.
	
	We can assume that $\tilde{\rho}\vert_{J} = \rho$. As an admissible $k$-order with prefix $J$, $\tilde{\rho}$ corresponds to a maximal chain in $B(n,k-1)$ which passes through $J$. Map $\rho$ to the path from $\emptyset$ to $J$ taken by $\tilde{\rho}$. 
	
	The inverse map is contained in Lemma \ref{DGH:lem:0JPaths}. Just map any path to its induced admissible $J$-order.
\end{proof}

\begin{notation}
	In light of Theorem \ref{DGH:thm:AdmPathCorr}, we will use $A_J(n,k)$ and $\paths{\emptyset}{J}{k-1}$ interchangeably depending upon what we want to emphasize.
\end{notation}

\begin{corollary}\label{DGH:cor:Lifting} Let $r \in \epaths{\emptyset}{J}{k-1}$. Then there exists an $\tilde{r} \in B(n,k)$ with $\Inv(r) = \Inv(\tilde{r})$.
\end{corollary}

\begin{proof}
	By Theorem \ref{DGH:thm:InvReal}, $\Inv(r)$ is a realizable $(k+1)$-set. By Theorem \ref{DGH:thm:Zieg}, there exists $\tilde{r} \in B(n,k)$ with $\Inv(\tilde{r}) = \Inv(r)$. 
\end{proof}

\subsection{Generalizing Ziegler's isomorphism}

We conclude this section by generalizing Theorem \ref{DGH:thm:Zieg} to the poset $\epaths{\emptyset}{J}{k-1}$. When we set $J = C(n,k)$, we recover Ziegler's original theorem \cite[Theorem 4.1B]{DGH:Z1993}.

\begin{theorem}\label{DGH:thm:PathZEquiv} Let $1 < k \leq n$. There is a natural isomorphism of posets between 
		\begin{enumerate}
	
		\item the poset $\epaths{\emptyset}{J}{k-1}$ equipped with $\leq_{\MS}$ and,
		
		\item the set of all realizable $(k+1)$-sets between $J_s$ and $J_s \cup J_F$, ordered by single step inclusion.
	
	\end{enumerate}
The isomorphism is given by the map $r \mapsto \Inv(r)$ for any $r \in \epaths{\emptyset}{J}{k-1}$.
\end{theorem}

\begin{proof}

	We first show that the mapping $r \mapsto \Inv(r)$ is injective. This is equivalent to showing that each equivalence class in $\epaths{\emptyset}{J}{k-1}$ is determined by its inversion set. Take $r, r' \in \epaths{\emptyset}{J}{k-1}$ with $r \neq r'$. Fix representatives $\gamma_{r} \in r$ and $\gamma_{r'} \in r'$. Both $\gamma_r$ and $\gamma_{r'}$ are paths from the $\emptyset$ to $J$ in $B(n,k-1)$. Pick a path $\tau$ in $B(n,k-1)$ from $J$ to $C(n,k)$. By concatenating $\tau$ on top of $\gamma_r$ and $\gamma_{r'}$ we create two paths from $\emptyset$ to $C(n,k)$ in $B(n,k-1)$. Label these new paths as $\tau \circ \gamma_{r}$ and $\tau \circ \gamma_{r'}$ respectively. 
	
	Consider the $(k+1)$-set\footnote{Recall Remark \ref{DGH:rmk:InvTop}} $\Inv(\tau)_{\emptyset} \coloneqq J_{\emptyset} \cap \Inv(\tau)$. We claim that 
	\begin{align}\label{DGH:eq:GammaInv}
		\Inv(\tau \circ \gamma_r) &= \Inv(\gamma_r) \cup \Inv(\tau)_{\emptyset} \\
		\Inv(\tau \circ \gamma_{r'}) &= \Inv(\gamma_{r'}) \cup \Inv(\tau)_{\emptyset}
	\end{align}
	It is clear that $\Inv(\gamma_{r}) \subseteq \Inv(\tau\circ \gamma_r)$. As $J$ is realizable, we know that $J_s \cup J_p \cup J_F \cup J_{\emptyset} = C(n,k+1)$. The order induced by $\tau \circ \gamma_r$ on any $k$-packet generated by $X \in J_s \cup J_p \cup J_F$ is determined by the order induced by $\gamma_r$. So any additions to the $\Inv(\tau \circ \gamma_r)$ must come from $\tau$ inducing the antilexicographic order on a $k$-packet generated by a $Y \in J_{\emptyset}$. The same holds for $\gamma_{r'}$ and $\tau \circ \gamma_{r'}$.
	
	Under the Manin-Schechtman correspondence (Theorem \ref{DGH:thm:MSthm}(\ref{MSthm:MSCor})), we can interpret $\tau \circ \gamma_r$ and $\tau \circ \gamma_{r'}$ as admissible $k$-orders. Because $\gamma_r$ and $\gamma_{r'}$ are not equivalent to each other, neither are $\tau \circ \gamma_r$ and $\tau \circ \gamma_{r'}$. From Theorem \ref{DGH:thm:Zieg}, we conclude that $\Inv(\tau\circ\gamma_r) \neq \Inv(\tau\circ\gamma_{r'})$. In light of \eqref{DGH:eq:GammaInv}, this can only occur if $\Inv(\gamma_r) \neq \Inv(\gamma_{r'})$. Hence $\Inv(r) = \Inv(\gamma_r) \neq \Inv(\gamma_{r'}) = \Inv(r')$ as desired.
	
	The mapping $r\mapsto \Inv(r)$ is also surjective. Take any realizable $(k+1)$-set $U$ for which $J_s \subseteq U \subseteq J_s \cup J_F$. According to Theorem \ref{DGH:thm:Zieg}, it must be the inversion set of some $\tilde{r}\in B(n,k)$. Because of Lemma \ref{DGH:cor:Zequiv}, there must be some admissible $k$-order $\rho \in \tilde{r}$ for which $J$ is a prefix. Again using the Manin-Schechtman correspondence, $\rho$ corresponds to a maximal chain in $B(n,k-1)$ which passes through $J$. If $\gamma \in \epaths{\emptyset}{J}{k-1}$ is the path taken by this maximal chain up to $J$, then $\Inv(\gamma) = U$. We get equality because $U$ contains no elements of $J_{\emptyset}$.
	
	All that is left is to show that $r \leq_{\MS} r'$ if and only if $\Inv(r) \leq_{\SSO} \Inv(r')$ for any pair $r, r' \in \epaths{\emptyset}{J}{k-1}$. Let us begin by assuming that $r \leq_{\MS} r'$. By Corollary \ref{DGH:cor:Lifting}, there exist $\tilde{r}, \tilde{r'} \in B(n,k)$ for which $\Inv(\tilde{r}) = \Inv(r)$ and $\Inv(\tilde{r'}) = \Inv(r')$. Furthermore, Lemma \ref{DGH:lem:Zequiv} allows us to select representatives $\rho \in \tilde{r}$ and $\rho' \in \tilde{r'}$ for which $J$ is a prefix. Again, we may view $\rho$ and $\rho'$ as paths in $B(n,k-1)$ from $\emptyset$ to $C(n,k)$ which pass through $J$. ``Decompose'' $\rho$ into a path $\gamma$ from $\emptyset \rightarrow J$ and $\tau$ from $J \rightarrow C(n,k)$, $\rho = \tau \circ \gamma$. Likewise, let $\rho' = \tau' \circ \gamma'$. 
	
	By definition, $\Inv(r)$ and $\Inv(r')$ differ by some subset of $J_F$. Recalling \eqref{DGH:eq:GammaInv}, we conclude that $\Inv(\tau)_{\emptyset} = \Inv(\tau') = \emptyset$, and $\Inv(r) = \Inv(\gamma_r)$ as well as $\Inv(r') = \Inv(\gamma_{r'})$. Thus $\gamma \in r$ and $\gamma_{r'} \in r'$. So the same sequence of packet flips which make $r \leq_{\MS} r'$ make $\tilde{r} \leq_{\MS} \tilde{r'}$. Using the isomorphism in Theorem \ref{DGH:thm:Zieg}, we have that $\Inv(r) = \Inv(\tilde{r}) \leq_{\SSO} \Inv(\tilde{r'}) = \Inv(r')$.
	 
	Now assume that $U_1$ and $U_2$ are realizable $(k+1)$-sets for which $J_s \subseteq U_i \subseteq J_s \cup J_F$ for $i \in \{1,2\}$. Further assume that $U_1 \leq_{\SSO} U_2$. Again, there exists maximal chains $\rho_1$ and $\rho_2$ in $B(n,k-1)$ for which $\Inv(\rho_1) = U_1$ and $\Inv(\rho_2) = U_2$. Theorem \ref{DGH:thm:Zieg} tells us that $\rho_1 \leq_{\MS} \rho_2$. As $\Inv(\rho_1)$ and $\Inv(\rho_2)$ differ only by elements in $J_F$, the sequence of packet flips consists entirely of packets generated by elements in $J_F$.
	 
	Arguing as before, we may use $\rho_1$ and $\rho_2$ to produce paths $\gamma_1, \gamma_2 \in \paths{\emptyset}{J}{k-1}$ for which $\Inv(\gamma_1) = \Inv(\rho_1)$, $\Inv(\gamma_2) = \Inv(\rho_2)$, $\gamma_1 \in r$, and $\gamma_2 \in r'$. Because $\rho_1$ and $\rho_2$ only differ by a sequence of packet flips from $J_F$ as well as elementary equivalences, $\gamma_1$ and $\gamma_2$ also only differ by a sequence of elementary equivalences and packet flips from $J_F$. Therefore, $r \leq_{\MS} r'$.  
\end{proof}

Although the proof of Theorem \ref{DGH:thm:PathZEquiv} does not proceed exactly along these lines, the key concept is that any path $\gamma \in \paths{\emptyset}{J}{k-1}$ may be extended to a maximal chain from $\emptyset$ to $C(n,k)$ in $B(n,k-1)$. According to Equation \eqref{DGH:eq:GammaInv}, the inversion set of this maximal chain is well-behaved. Not only is the inversion set well-behaved, but it can be made to equal the inversion set of $\gamma$!

\begin{lemma}\label{DGH:lem:ExtLem}
	If $\gamma \in \paths{\emptyset}{J}{k-1}$, then there exists a path $\tau$ in $B(n,k-1)$ from $J$ to $C(n,k)$ for which the maximal chain $\tau \circ \gamma$ has $\Inv(\tau \circ \gamma) = \Inv(\gamma)$. 
\end{lemma}
\begin{proof} 
	The inversion set of $\gamma$ is a realizable $(k+1)$-set. By Theorem \ref{DGH:thm:Zieg}, it must be the inversion set of some $\tilde{r} \in B(n,k)$. Because of Lemma \ref{DGH:lem:Zequiv}, there is a representative $\rho \in \tilde{r}$ for which $J$ is a prefix. Viewing $\rho$ as a maximal chain in $B(n,k-1)$, we may decompose it into $\tilde{\gamma}$, a path from $\emptyset$ to $J$, and $\tau$, a path from $J$ to $C(n,k)$. Equations \eqref{DGH:eq:GammaInv} show that $\Inv(\tau)_{\emptyset} = \emptyset$. Therefore $\tau \circ \gamma$ is a maximal chain in $B(n,k-1)$ for which $\Inv(\gamma) = \Inv(\tau \circ \gamma)$. 
\end{proof}
 

\begin{theorem}\label{DGH:thm:PathSnk}
	Whenever $1 < k \leq n$, the poset $\epaths{\emptyset}{J}{k-1}$ has a unique minimal element $r_{\min}$ and a unique maximal element $r_{\max}$. These elements are determined by their inversion sets which are, respectively, $J_s$ and $J_s \cup J_F$.
\end{theorem}
\begin{proof}
	We begin by proving that $\epaths{\emptyset}{J}{k-1}$ has elements $r_{\min}$ and $r_{\max}$ with $\Inv(r_{\min}) = J_s$ and $\Inv(r_{\max}) = J_s \cup J_F$. If $r_{\min}$ exists, then it is clearly minimal because $J_s$ is minimal under the single-step inclusion order found in Theorem \ref{DGH:thm:PathZEquiv}. Likewise, $r_{\max}$ is clearly maximal.
	
	In order to prove that $r_{\min}$ exists, we need to find an admissible $J$ order $\rho$ with $\Inv(\rho) = J_s$. Then $r_{\min} = \pi(\rho) \in \epaths{\emptyset}{J}{k-1}$. Fortunately, $J_s$ is a realizable $(k+1)$-set. By Theorem \ref{DGH:thm:Zieg} there exists $r \in B(n,2)$ for which $\Inv(r) = J_s$. By Lemma \ref{DGH:lem:Zequiv}, there exists $\rho \in A(n,k)$ with $\Inv(\rho) = J_s$ and for which $J$ is a prefix. The order $\rho\vert_{J}$ is the desired admissible $J$ order. The same argument is used to prove that $r_{\max}$ exists. 	
	
	We now prove that $r_{\min}$ is the unique minimal element in $\epaths{\emptyset}{J}{k-1}$. This is accomplished by by proving that every realizable $(k+1)$-set $K$ where $J_s \subseteq K \subseteq J_s\cup J_F$ obeys $J_s \leq_{SS} K$. 
	
	Lemma \ref{DGH:lem:SSrel} equates $J_s \leq_{SS} K$ with finding an admissible $(k+1)$-order for which $J_s$ and $K$ are prefixes. This, in turn, is equivalent to finding some admissible $(k+1)$-order $\rho \in A(n,k+1)$ for which $((J_s)_s \cup K_s) \subseteq \Inv(\rho)$ and $((J_s)_p \cup K_p) \subseteq \Inv(\rho)^c$. Alternatively, we may find some $r \in B(n,k+1)$ which satisfies the same requirements. In particular, $\Inv(r) = (J_s)_s \cup K_s$ satisfies the requirement.
	
	According to Theorem \ref{DGH:thm:Zieg} such a $r$ exists whenever $(J_s)_s \cup K_s$ is a realizable $(k+2)$-set. We claim that $(J_s)_s \cup K_s = K_s$ and is therefore realizable by Lemma \ref{DGH:lem:Jreal}. Take $X \in (J_s)_s$. Since $J_s \subseteq K$, a suffix of $X$ is contained in $K$, so $X \in K_s \cup K_F$. We need only prove that $X \in K_F$ produces a contradiction. Since $K\setminus J_s \subseteq J_F$, such a packet $P_X$ would have segmentation \eqref{badsets:Fs}, which is forbidden by Corollary \ref{DGH:cor:badsets}. Hence $X$ cannot be in $K_F$, and $X \in K_s$.
		
	The same strategy is used to prove that $r_{\max}$ is the unique maximal element of $\epaths{\emptyset}{J}{k-1}$. We need to show that $K_s \cup (J_s \cup J_F)_s$ is realizable. It is sufficient to show that $K_s \cup (J_s \cup J_F)_s = K_s$.
	
	Let $X \in (J_s \cup J_F)_s$. The first element of $P_X$ is either in $J_p$ or $J_\emptyset$, while the last element of $P_X$ is either in $J_s$ or $J_F$. If the last element of $P_X$ is in $J_s$, then $X \in K_s$. If the last element of $P_X$ is in $J_F$, then Lemma \ref{DGH:lem:parts} shows that the first element is in $J_p$. According to Lemma \ref{DGH:lem:parts}, the only possible segmentation of $P_X$ is 
	\[\{P_X \cap J_p < P_X \cap J_F\}.\]
	But this is impossible by Lemma \ref{DGH:cor:badsets}.  
\end{proof}

Theorem \ref{DGH:thm:PathZEquiv} and Theorem \ref{DGH:thm:PathSnk} justify the following definition.

\begin{definition}\label{DGH:def:JBruhatOrd}
	The poset $\epaths{\emptyset}{J}{k-1}$ is called the \emph{second Bruhat order} of $J$.
\end{definition}

\section{Higher Bruhat orders for realizable $2$-sets}\label{DGH:Sec:GenBruhat}



The purpose of this section is to define an analog of higher Bruhat orders for an arbitrary realizable $2$-set. Drawing inspiration from Theorem \ref{DGH:thm:PathZEquiv}, the analog of the $k$-th higher Bruhat order will be a collection of intervals in $B(n,k)$. They will have unique minimal and maximal elements, again distinguished by their inversion sets. These subposets will also be defined inductively so that they automatically satisfy a generalization of the Manin-Schechtman correspondence. Before establishing the full theory, it is useful to study how the first step in this process works. Because the second higher Bruhat order for any realizable $k$-set $J$ has a unique minimal and maximal element, we begin by describing this first step for $\epaths{\emptyset}{J}{k-1}$.

\subsection{A higher Bruhat order for paths in $\epaths{\emptyset}{J}{k-1}$} \label{DGH:ssec:BPaths}


Theorem \ref{DGH:thm:PathSnk} allows us to discuss maximal chains in $\epaths{\emptyset}{J}{k-1}$. Recall that the minimal element has inversion set $J_s$ and the maximal element has inversion set $J_s \cup J_F$. Because elements of $\epaths{\emptyset}{J}{k-1}$ are ordered by the single step inclusion order, maximal chains in $\epaths{\emptyset}{J}{k-1}$ correspond to a subset of the set of total orders on $J_F$. Call this subset of total orders $\mathcal{O}_F$. Because $J_F$ is often not realizable, it does not make sense to talk about the entries of $\mathcal{O}_F$ as admissible. 

However, we can extend any order $\rho_F \in \mathcal{O}_F$ to an admissible $(k+1)$-order on $J_s \cup J_F$. We may do so by fixing an admissible $J_s$ order $\rho_s \in A_{J_s}(n,k+1)$. We then pre-append $\rho_s$ to every element in $\mathcal{O}_F$. This creates a collection of total orders on $J_s \cup J_F$ which we denote $\extOrd{\rho_s}{S\cup F} = \lbrace \rho_s\rho_F \mid \rho_F \in \mathcal{O}_F\rbrace$. 

\begin{lemma}\label{DGH:lem:OextA} Every element in $\extOrd{\rho_s}{S\cup F}$ is an admissible $J_s \cup J_F$ order. More precisely, it is a subset of the following subset of admissible $(J_s\cup J_F)$ orders:
\begin{equation}\label{DGH:eqn:AdmJSJF}
	\mathcal{A}_{S\leq F} = \lbrace \rho \in A_{J_s\cup J_F}(n,k+1) \mid s \leq_{\rho} f \text{ for all } x \in J_s, f \in J_F\rbrace
\end{equation}
\end{lemma}

\begin{proof}
The only non-trivial part of the claim is that each order in $\extOrd{\rho_s}{S\cup F}$ is admissible. By Theorem \ref{DGH:thm:AdmPathCorr}, we can instead show that any $\rho \in \extOrd{\rho_s}{S\cup F}$ comes from some path in $B(n,k-1)$ from $\emptyset$ to $J_s \cup J_F$. This in turn is equivalent to producing an ascending chain of realizable $(k+1)$-sets from $\emptyset$ to $J_s \cup J_F$ under single step inclusion.

By design, the $\rho_s$ component of $\rho$ corresponds to a path from $\emptyset$ to $J_s$. This automatically produces an ascending chain from $\emptyset$ to $J_s$. The $\rho_F$ component also produces an ascending chain from $J_s$ to $J_s \cup J_F$ according to Theorem \ref{DGH:thm:PathZEquiv}. Stacking these two chains gives the desired ascending chain from $\emptyset$ to $J_s \cup J_F$. Since the chain contains $J_s$, we know that $\extOrd{}{S\cup F}$ is contained in $\mathcal{A}_{S\leq F}$.
\end{proof}

In order to produce an admissible $(J_s \cup J_F)$ order out of a maximal chain in $\epaths{\emptyset}{J}{k}$, we had to pick an admissible $J_s$ order. This choice ultimately does not matter, at least up to elementary equivalences. The general idea is that the choice of $\rho_s$ does not alter the resulting inversion sets. The proof depends upon this next lemma. 

\begin{lemma}\label{DGH:lem:JSF}
	Let $J$ be a realizable $k$-set. Then $(J_s)_{F} = \emptyset$.
\end{lemma}
\begin{proof}
	The set $(J_s)_F$ is a $(k+2)$-set so let $X = \{1,\dots,k+2\}$. Suppose that $\widehat{1} \in J_s$. Then $\widehat{1}\widehat{k+2} \notin J$. This is the lexicographic maximal element in $\widehat{k+2}$. So $\widehat{k+2} \in J_p \cup J_{\emptyset}$. We conclude that $X \notin (J_s)_F$.
\end{proof}

\begin{corollary}\label{DGH:cor:rsequiv}
	All admissible $J_s$ orders are equivalent to each other.
\end{corollary}
\begin{proof}
	By Theorem \ref{DGH:thm:PathZEquiv}, admissible $J_s$ orders are equivalent to each other if they have the same inversion set. Every admissible $J_s$ order has the same inversion set because $(J_s)_F = \emptyset$.
\end{proof}

\begin{lemma}\label{DGH:lem:choice}
	If $\rho_s$ and $\rho'_s$ are both admissible $J_s$ order, then 
\begin{equation}\label{DGH:eq:choice}
	\pi(\extOrd{\rho_s}{S\cup F}) = \pi(\mathcal{O}^{\rho'_s}_{S\cup F}),
\end{equation}
	where $\pi: \paths{\emptyset}{J_s\cup J_F}{k} \rightarrow \epaths{\emptyset}{J_s\cup J_F}{k}$ is the quotient map defined in Definition \ref{DGH:def:EquivPaths}.
\end{lemma}
\begin{proof}
	We know that $\rho_s$ and $\rho'_s$ are equivalent to each other from Corollary \ref{DGH:cor:rsequiv}. As $\rho'_s\rho_F$ and $\rho_s\rho_F$ differ only in how $\rho_s$ and $\rho'_s$ differ, it follows that $\rho'_s\rho_F$ and $\rho_s\rho_F$ are also equivalent. So they have the same inversion set. 
	
	Hence $\pi(\extOrd{\rho_s}{S\cup F})$ and $\pi(\extOrd{\rho'_s}{S\cup F})$ give rise to the same collection of inversion sets. Because $\pi(\extOrd{\rho_s}{S\cup F}) \subseteq \epaths{\emptyset}{J_s\cup J_F}{2}$, its entries are uniquely determined by their inversion sets. The same is true for $\pi(\extOrd{\rho'_s}{S\cup F})$. We conclude that \eqref{DGH:eq:choice} holds.	 
\end{proof}

Since every admissible $(J_s\cup J_F)$ order which orders $J_s$ before $J_F$ automatically induces an admissible $J_s$ order, we may generalize Lemma \ref{DGH:lem:choice} a bit further in the following corollary. This corollary allows us to define the third Bruhat order of $J$.

\begin{corollary}\label{DGH:cor:choice}
	The following subposets of $\epaths{\emptyset}{J_s\cup J_F}{2}$ are equal:
	\begin{equation}
		\pi(\extOrd{\rho_s}{S\cup F}) = \pi(\mathcal{A}_{S\leq F})
	\end{equation}
\end{corollary}

\begin{definition}\label{DGH:def:J1Higher}
	Let $J$ be a realizable $k$-set. The third Bruhat order of $J$ is defined as
	\begin{equation}
		\HB{3}{J} \coloneqq \pi(\mathcal{A}_{S\leq F})
	\end{equation}
	By definition, $\HB{3}{J} \subseteq \epaths{\emptyset}{J_s \cup J_F}{k}$. 
\end{definition}

We proved in Theorem \ref{DGH:thm:PathSnk}, that $\epaths{\emptyset}{J_s\cup J_F}{k}$ is a ranked poset with unique minimal and maximal elements. The minimal element is uniquely determined by its inversion set, which is $(J_s \cup J_F)_s$. The maximal element is uniquely determined by the inversion set $(J_s\cup J_F)_s \cup (J_s\cup J_F)_F$. 

The third Bruhat order of $J$ inherits a ranked poset structure from $\epaths{\emptyset}{J_s \cup J_F}{k}$. We also claim that $\HB{3}{J}$ has a unique minimal element. The unique minimal element of $\HB{3}{J}$ coincides with the unique minimal element of $\epaths{\emptyset}{J_s\cup J_F}{k}$. It has $(J_s \cup J_F)_s$ as its inversion set. In general, we do not know if $\HB{3}{J}$ has a unique maximal element. If it does then the unique maximal element of $\HB{3}{J}$ cannot coincide with the unique maximal element of $\epaths{\emptyset}{J_s\cup J_F}{2}$. Instead, the inversion set of the maximal element of $\HB{3}{J}$ must be a subset of: 
\begin{equation}\label{DGH:eq:maxinv}
(J_s\cup J_F)_s \cup (J_F)_F.
\end{equation}
This asymmetry arises because $(J_s \cup J_F)_s \cup (J_s \cup J_F)_F$ often intersects non-trivially with $(J_s)_p$. When this occurs, any admissible $(J_s\cup J_F)$ order $\rho$ with $\Inv(\rho) = (J_s \cup J_F)_s \cup (J_s\cup J_F)_F$ is no longer equivalent to an order that orders the elements of $J_s$ before the elements of $J_F$. In particular, it fails the criterion laid out in Lemma \ref{DGH:lem:SSrel}. At the moment we are unable to prove that the realizable set in \eqref{DGH:eq:maxinv} is maximal when $J$ is an arbitrary $k$-set. However, we are able to prove it when $J$ is a realizable $2$-set. At this point, we eschew further discussion of the minimal and maximal elements of $\HB{3}{J}$ in favor of presenting the general theory for $k = 2$. 

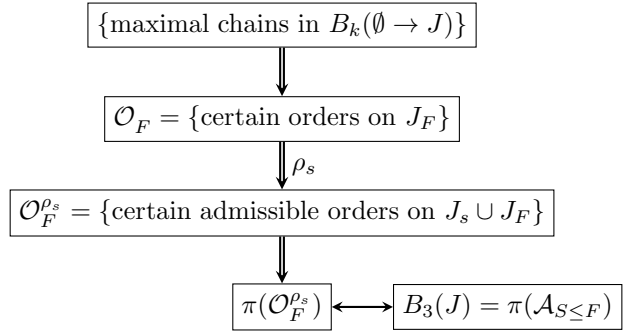
\begin{figure}
	\centering
	\begin{tikzpicture}
		\node[rectangle, draw] (MC) at (1.5,7){$\lbrace \text{maximal chains in } \epaths{\emptyset}{J}{k}\rbrace$};
		\node[rectangle, draw] (OF) at (1.5,5.75){$\extOrd{}{F} = \lbrace\text{certain orders on } J_F\rbrace$};
		\node[rectangle, draw] (EOF) at (1.5,4.5){$\extOrd{\rho_s}{F} = \lbrace\text{certain admissible orders on } J_s\cup J_F\rbrace$};
		\node[rectangle, draw] (PEOF) at (1.5,3.25){$\pi(\extOrd{\rho_s}{F})$};
		\node[rectangle, draw] (HB) at (4.5,3.25){$\HB{3}{J} = \pi(\mathcal{A}_{S\leq F})$};

		\draw[-stealth,double, thick] (MC.south) -- (OF.north);
		\draw[-stealth,double, thick] (OF.south) -- (EOF.north) node[right, xshift = 0cm, yshift = .3cm]{$\rho_s$};
		\draw[-stealth,double, thick] (EOF.south) -- (PEOF.north);
		\draw[stealth-stealth,thick] (PEOF.east) -- (HB.west);
		
	\end{tikzpicture}
\caption{A simple flow chart outlining how $\HB{3}{J}$ was constructed.}
\end{figure}

\subsection{General Result when $k = 2$}\label{DGH:ssec:GenRes}

\begin{rmk}
	Throughout this section assume that $J$ is a realizable $2$-set.
\end{rmk}

\begin{definition}
	Let $J$ be a realizable $2$-set. Recursively define the following sets
	\begin{center}
		\begin{tabular}{cc}
			$L^2 = \emptyset$ & $M^2 = J$ \\
			$L^3 = J_s$ & $M^3 = J_s \cup J_F$ \\
			$L^4 = (M^3)_{s}$ & $M^4 = (M^3)_s \cup (M^3 \setminus L^3)_{F}$ \\
			$\vdots$&$\vdots$ \\
			$L^i = (M^{i-1})_s$ & $M^i = (M^{i-1})_{s} \cup (M^{i-1}\setminus L^{i-1})_{F}$
		\end{tabular}
	\end{center}
	The \emph{i-th Bruhat order of $J$} is defined as the set
	\begin{equation}
		\HB{i}{J} = \pi\Big(\lbrace\text{admissible orders on } M^i\text{ which place }L^i \text{ first}\rbrace\Big)
	\end{equation}
	The i-th Bruhat order $B_i(J)$ inherits a partial order as a subposet of $\epaths{\emptyset}{M^i}{i-1}$.
\end{definition}

\begin{rmk}\label{DGH:rmk:SecondBruhat}
	Whenever $L^i = \emptyset$, we get that $\HB{i}{J} = \epaths{\emptyset}{M^i}{i-1}$. In particular, this is true when $i = 2$ and $\HB{2}{J} = \epaths{\emptyset}{J}{1}$.
\end{rmk}

\begin{theorem}\label{DGH:thm:DGHBruhat}
	Let $J$ be a realizable $2$-set. The i-th Bruhat order $B_i(J)$ has the following properties.
		\begin{enumerate}
		\item It is a ranked poset under the the relation $\leq_{\MS}$. The rank function is given by $\vert\Inv(\cdot)\vert$. \label{DGHBruhat:partone}
		
		\item As a ranked poset, $\HB{i}{J}$ has a unique minimal element $r_{\min}$ and unique maximal element $r_{\max}$. Their inversion sets are $\Inv(r_{\min}) = L^{i+1}$ and $\Inv(r_{\max}) = M^{i+1}$. \label{DGHBruhat:parttwo}
		
		\item \label{DGHBruhat:MSCor} Maximal chains in $B_i(J)$ correspond to a subset of orders on $(M^{i+1} \setminus L^{i+1})$. These maximal chains can be extended to admissible orders on $M^{i+1}$ which order $L^{i+1}$ before $(M^{i+1}\setminus L^{i+1})$. Up to elementary equivalences, these are all such admissible orders on $M^{i+1}$ which order $L^{i+1}$ first.
		This gives a surjection between maximal chains in $\HB{i}{J}$ and the poset $\HB{i+1}{J}$.
		
		\item Each element of $\HB{i}{J}$ is uniquely defined by its inversion set. \label{DGHBruhat:partfour}
		\end{enumerate}
\end{theorem}

Parts \ref{DGHBruhat:partone} and \ref{DGHBruhat:partfour} are true because they are true for $\epaths{\emptyset}{M^i}{i-1}$. As a subposet, $B_i(J)$ inherits these properties. Part \ref{DGHBruhat:MSCor} was explained in detail for the second Bruhat order in Subsection \ref{DGH:ssec:BPaths}. The process works exactly the same for higher Bruhat orders. Given an order on $(M^{i+1}\setminus L^{i+1})$ arising as a maximal chain, we can extend this to an admissible order on $L^{i+1}$ by pre-appending with any admissible order on $L^{i+1}$. The resulting order on $M^{i+1}$ must be admissible. It corresponds to some element in $\paths{\emptyset}{M^{i+1}}{i}$, which is in bijection with admissible orders on $M^{i+1}$ by Theorem \ref{DGH:thm:AdmPathCorr}.

Part \ref{DGHBruhat:parttwo} requires additional work. First, we must show that $L^i$ and $M^i$ are realizable $i$-sets. Then we must show that they are in fact the inversion sets of minimal and maximal elements in $B_i(J)$. This is not as difficult for $L^i$, so we tackle it first. 

\begin{rmk}
	In general, all of the proceeding proofs rely on an inductive hypothesis that $L^{k}$ and $M^{k}$ are realizable for all $3 \leq k \leq i-1$. The base case when $k = 2$ is trivially true as $L^2 = \emptyset$ and $M^2 = J$. 
\end{rmk}

\begin{lemma}
	For $i > 2$, $L^i$ is realizable if $M^{i-1}$ is realizable.
\end{lemma} 
\begin{proof}
	As $L^i = M^{i-1}_s$, it is realizable by Lemma \ref{DGH:lem:Jreal}. 
\end{proof}

\begin{lemma}\label{DGH:lem:unimin}
	The element $r_{\min} \in B_{i-1}(J)$ with inversion set $L^i$ is the unique minimal element in $\HB{i-1}{J}$.
\end{lemma}
\begin{proof}
	By Theorem \ref{DGH:thm:PathSnk} we already know that the unique minimal element $r_{\min}$ of $\epaths{\emptyset}{M^{i-1}}{i-2}$ has $\Inv(r_{\min}) = (M^{i-1})_s$. We just need to show that $r_{\min} \in \HB{i-1}{J}$. This is true if $L_p^{i-1} \cap \Inv(r_{\min}) = \emptyset$ and $L^{i-1}_s \subseteq \Inv(r_{\min})$.
	
	Lemma \ref{DGH:lem:ABpart} with $A = L^{i-1}$ and $B = M^{i-1}$ tells us that 
	\begin{align*}
		L^i &\coloneqq M^{i-1}_s \\
		&= (M^{i-1}_s \cap L^{i-1}_s) \cup (M^{i-1}_s \cap L^{i-1}_{\emptyset}).
	\end{align*}
	This shows that $L^{i-1}_p \cap \Inv(r_{\min}) = \emptyset$. 
	
	Lemma \ref{DGH:lem:ABpart} also shows that 
	\begin{equation}\label{unimin:Ls1}
		L^{i-1}_s = (L^{i-1}_s \cap M^{i-1}_s) \cup (L^{i-1}_s \cap M^{i-1}_F).
	\end{equation}
	If $X \in L^{i-1}_s \cap M^{i-1}_F$, then because $L^{i-1} \subseteq M^{i-1}$, $P_X$ has the following segmentation (see proof of Lemma \ref{DGH:lem:J1SJ2F} for more detail)
	\begin{equation*}
		\{P_X \cap M^{i-2}_F < P_X \cap M^{i-2}_s\}.
	\end{equation*}
	Such segmentations are disallowed by Lemma \ref{DGH:cor:badsets}. Therefore Equation \eqref{unimin:Ls1} becomes
	\begin{equation}\label{unimin:Ls2}
		L^{i-1}_s = (L^{i-1}_s \cap M^{i-1}_s).
	\end{equation}
	This shows that $L^{i-1}_s \subseteq L^{i} \subseteq \Inv(r_{\min})$. 
	
	Hence $L^{i}$ satisfies the criterion in Lemma \ref{DGH:lem:Zequiv}. Any path with $L^i$ as its inversion set is equivalent to one which passes through $L^{i-1}$ and then $M^{i-1}$. Thus there exists an admissible $M^{i-1}$ order $\rho \in r_{\min}$ that puts $L^{i-1}$ before $(M^{i-1}\setminus L^{i-1})$ as desired.  
\end{proof}

The emptiness of $L^{i-1}_s \cap M_F^{i-1}$ played a crucial role in the proof of Lemma \ref{DGH:lem:unimin}. Because we will use it several more times, it is worth stating a slight generalization as a standalone result.

\begin{lemma}\label{DGH:lem:J1SJ2F}
	Let $U = J_s$ for a realizable $k$-set $J$ and let $U'$ be any realizable $(k+1)$-set satisfying $U \subseteq U' \subseteq J_s \cup J_F$. Then $U'_F \cap U_s = \emptyset$.
\end{lemma}
\begin{proof}
	The proof is exactly as laid out in Lemma \ref{DGH:lem:unimin}. We provide more detail here. Take $X \in U'_F \cap U_s$. By Lemma \ref{DGH:lem:JSF}, we know that $P_X$ must contain a mix of elements from $J_s$ and $J_F$. In particular, the lexicographic maximal element of $P_X$ must be in $J_s$, while the lexicographic minimal element of $P_X$ must be in $J_F$. According to Lemma \ref{DGH:lem:parts}, $P_X$ must have the following segmentation
	\begin{equation}
		\{P_X \cap J_F < P_X \cap J_s\}
	\end{equation}
	Such a segmentation is disallowed by Corollary \ref{DGH:cor:badsets}. Hence, $X$ does not exist and $U'_F \cap U_s = \emptyset$.
\end{proof}

So far we have repeatedly used Corollary \ref{DGH:cor:badsets} in order to rule out the existence of various kinds of packets. This was necessary when proving that $L^i$ was the inversion set of the unique minimal element of $\HB{i-1}{J}$. We will need to prove the non-existence of two more kinds of packets in order to prove that $M^i$ is the inversion set of the unique maximal element of $\HB{i-1}{J}$. The realizability of $M^i$ depends on showing that no $i$-packets have the following segmentation:
\begin{equation}\label{DGH:eqn:badLiMi}
	\{P_X \cap (L_p^{i-1} \cap M_F^{i-1}) < P_X \cap (L^{i-1}_{\emptyset} \cap M^{i-1}_{F}) < P_X \cap (L^{i-1}_{\emptyset} \cap M^{i-1}_{p})\}.
\end{equation}
While the maximality of $M^i$ in the single step inclusion order will depend on showing that there are no $(i+1)$-packets with the following segmentation:
\begin{equation}\label{DGH:eqn:badKsMs}
	\{K_{\emptyset} \cap M^i_s < K_{\emptyset} \cap M^i_F < K_s \cap M^i_F\}.
\end{equation}

Both proofs will critically depend upon the assumption that the $J$ which we use to define $L^i$ and $M^i$ is a realizable $2$-set. It will also depend upon the following observation. 

\begin{lemma}\label{DGH:lem:nobadobs} For any $X \in M^i\setminus L^i$, every $(i-1)$-subset of $X$ is in $M^{i-1}\setminus L^{i-1}$. Moreover, any $(i-l)$-subset of $X$ is in $M^{i-l}\setminus L^{i-l}$ for $1 \leq l \leq i-2$.
\end{lemma}
\begin{proof}
	By definition $M^i\setminus L^i = (M^{i-1}\setminus L^{i-1})_F$. So $P_X \subset M^{i-1}\setminus L^{i-1}$. Hence, every $(i-1)$-subset of $X$ is in $M^{i-1}\setminus L^{i-1}$. Repeatedly apply these definitions in order to prove that every $(i-l)$-subset of $X$ is in $M^{i-l}\setminus L^{i-l}$.  
\end{proof}

\begin{lemma}\label{DGH:lem:nobadLpMF}
	For $2 \leq i \leq n-1$ there are no $(i+1)$-packets with the following segmentation:
	\begin{equation}\label{DGH:eq:badLpMF}
		\{P_X \cap (L_p^{i} \cap M_F^{i}) < P_X \cap (L^{i}_{\emptyset} \cap M^{i}_{F}) < P_X \cap (L^{i}_{\emptyset} \cap M^{i}_{p})\}.
	\end{equation}
\end{lemma}
\begin{proof}
	Without loss of generality let $X = \lbrace 1,\ldots, i+2\rbrace$. Suppose that $P_X$ segmented as in Equation \eqref{DGH:eq:badLpMF}. Our ultimate goal is to arrive at a contradiction. In order to do so we must prove two subclaims. The first is that:
	\begin{align}\label{nobadLpMF:C1}
		\begin{split}
		\widehat{1}, \widehat{2} \in L^i_{\emptyset} \cap M^i_p \\
		\widehat{i+1}, \widehat{i+2} \in L^i_p\cap M^i_F.
		\end{split}
	\end{align}
	By assumption $\widehat{i+2} \in L^i_{p}\cap M^i_F$. This implies that $\widehat{i+1}\widehat{i+2} \in L^i$. As $\widehat{i+1}\widehat{i+2} \in \widehat{i+1}$, we have that $\widehat{i+1}\notin L^{i}_{\emptyset}$. Furthermore, $\widehat{i+1} \in M^i_F$, again by assumption. The only possibility, given the segmentation above, is that $\widehat{i+1} \in L^i_p \cap M^i_F$. Similar reasoning applied to $\widehat{1}$ establishes that $\widehat{2} \in M^i_p$ and thus $\widehat{2} \in L^i_{\emptyset} \cap M^i_p$.
	
	Let $2 \leq l \leq i-2$. The second subclaim is that
	\begin{align}
		Y_l &\coloneqq X\sdiff{i+2-l,\ldots,i+2} \in M_s^{i-l} = L^{i-l+1}, \text{ when l is odd}\label{nobadLpMF:odd} \\
		Y_l &\coloneqq X\sdiff{i+2-l,\ldots,i+2} \in L_p^{i-l} \cap M_F^{i-l} \text{ when l is even}\label{nobadLpMF:even}
	\end{align}	
	
	To prove this second subclaim we induct on $l$. When $l = 1$, the base case states that $\widehat{i+1}\widehat{i+2} \in L^i$, which was just proved. If $l$ is odd, then our induction hypothesis states that $Y_{l-1} \in L_p^{i-l} \cap M_F^{i-l}$. Of course, the lexicographic minimal element in the packet generated by $Y_{l-1}$
	is $Y_{l}$. Hence $Y_l \in L^{i-l}$ as desired. 
		
	Now suppose that $l$ is even. By our inductive hypothesis $Y_{l-1} \in M_s^{i-(l-1)}$. Hence $Y_l \notin M^{i-(l-1)}$ by the definition of $M^{i-(l-1)}$. Lemma \ref{DGH:lem:ABpart} combined with Lemmas \ref{DGH:lem:JSF} and \ref{DGH:lem:J1SJ2F} implies that: 
	\begin{equation}\label{nobadLpMF:evensub}
		Y_l \in M_p^{i-l} \cup M_{\emptyset}^{i-l} \cup (M_F^{i-l} \cap L_p^{i-l}).
	\end{equation}
	To prove our subclaim first remember that $\widehat{1} \in L^i_{\emptyset}\cap M^i_F$. So $\widehat{1}\widehat{i+2} \in M^i\setminus L^i$. Lemma \ref{DGH:lem:nobadobs} implies that every subset of $\widehat{1}\widehat{i+2}$ is in $M^k\setminus L^k$ for an appropriate $k$. Therefore, every packet of every subset of $\widehat{i+2}$ has its maximal element in an appropriate $M^k\setminus L^k$. In particular, this is true of for the packet generated by $Y_l$, whose maximal element is in $M^{i-l}$. It follows that $Y_l$ is not in $M_p^{i-1}$ nor $M^{i-l}_{\emptyset}$. In light of \eqref{nobadLpMF:evensub}, we conclude that $Y_l \in M_F^{i-l} \cap L_p^{i-l}$ as desired. This completes our induction.
	
	To complete the proof of this lemma we resort to analyzing the cases where $i$ is odd and where $i$ is even separately.
	\begin{description}
		\item[$i$ is odd] Set $l = i-2$. Because $l$ is odd, Claim 2 gives $\lbrace 1,2,3 \rbrace \in M_s^2 = J_s$. Therefore $\lbrace 1,2\rbrace \notin J$. Based on the segmentation in \eqref{DGH:eq:badLpMF} there must be some $j$ with $\widehat{j} \in L_{\emptyset}^i \cap M_F^i$. We know that $2 < j < i+1$ due to \eqref{nobadLpMF:C1}. We know that $\widehat{j} \in M^i \setminus L^i$. As $2 < j < i+1$, Lemma \ref{DGH:lem:nobadobs} implies that $\lbrace 1,2\rbrace \in J$; a contradiction.
		\item[$i$ is even] Set $l = i-2$. Because $l$ is even, Claim 2 gives $\lbrace 1,2,3\rbrace \in M^2_F\cap L_p^2$. But $L^2 = \emptyset$, so $L^2_p = \emptyset$. This leads to an obviously nonsensical result.
	\end{description}
\end{proof}

\begin{lemma}\label{DGH:lem:nobadK0MF}
	Let $i \geq 2$. Pick any realizable $i$-set $K$ where $L^i \subseteq K \subseteq M^i$. No $i$-packet $P_X$ is segmented as:
	\begin{equation}\label{nobadK0MF:eqn:seg}
		\{K_{\emptyset}\cap M^i_s < K_{\emptyset}\cap M^i_F < K_s \cap M^i_F\}
	\end{equation} 
\end{lemma}
\begin{proof}
	The proof is very similar to the proof of Lemma \ref{DGH:lem:nobadLpMF}. We again argue by way of contradiction. Assume that there is an $(i+1)$-packet $P_X$ segmented as in Equation \eqref{DGH:eqn:badKsMs}. 
	
	The following analog to \eqref{nobadLpMF:C1} holds:
	\begin{align}\label{nobadK0MF:C1}
		\begin{split}
			\widehat{1}, \widehat{2} \in K_s \cap M^i_F \\
			\widehat{i+1}, \widehat{i+2} \in K_{\emptyset} \cap M^i_s.
		\end{split}
	\end{align}
	The proof is similar and left to the reader.
		
	By assumption $\widehat{1} \in K_s \cap M^i_F$. As the lexicographic minimal subset of $\widehat{1}$ of size $i$, we know that
	\begin{equation}
		\widehat{1}\widehat{i+2} \in M\setminus K \subseteq (M^{i-1}\setminus L^{i-1})_F.
	\end{equation}
	Lemma \ref{DGH:lem:nobadobs} tells us that every $(i-1-l)$ subset of $\widehat{1}\widehat{i+2}$ is in $M^{i-1-l}$.  
	
	As in the proof of Lemma \ref{DGH:lem:nobadLpMF} we claim the following for all $1 \leq l \leq i-2$:
	\begin{align}
		X&\sdiff{i+2-l,\ldots,i+2} \in M_s^{i-l} = L^{i-l+1} \tag{\text{odd l}} \\
		X&\sdiff{i+2-l,\ldots,i+2} \in L^{i-l}_p \cap M^{i-l}_F \tag{\text{even l}}.
	\end{align}
	The proof once again proceeds via induction. The base case is when $l = 1$. When $l = 1$, we must show that $\widehat{i+1}\widehat{i+2} \in M_s^{i-1}$. Note that $\widehat{i+1}\widehat{i+2}$ is the lexicographic minimal element in $\widehat{i+2}$. Additionally, we have assumed that $\widehat{i+2} \in K_{\emptyset} \cap M^i_s$. Therefore, $\widehat{i+1}\widehat{i+2} \notin M^i$ by definition. Thus Lemma \ref{DGH:lem:ABpart} combined with Lemmas \ref{DGH:lem:JSF} and \ref{DGH:lem:J1SJ2F} shows that $\widehat{i+1}\widehat{i+2} \in M^{i-1}_p \cup M^{i-1}_{\emptyset}\cup (M_{F}^{i-1}\cap L^{i-1}_p)$. Due to our earlier observation, $\widehat{1}\widehat{i+1}\widehat{i+2}$, the lexicographic maximal $i$-subset of $\widehat{i+1}\widehat{i+2}$, is in $M^{i}$. Thus $\widehat{i+1}\widehat{i+2} \in M^{i-1}_F \cap L^{i-1}_p$ as desired. The remainder of the inductive proof proceeds exactly as it did in Lemma \ref{DGH:lem:nobadLpMF}.
	
	In fact, the remainder of this proof is identical to the end of the proof in Lemma \ref{DGH:lem:nobadLpMF}. The same contradictions are reached when we set $l = i-2$. 
\end{proof}

\begin{lemma}\label{DGH:lem:Mireal}
	The set $M^i$ is a realizable $i$-set for $i \geq 2$.
\end{lemma}
\begin{proof}
	As in the proof that $L^i$ is realizable, we proceed via induction. The base case is when $i = 2$. This is true by assumption as $M^2 \coloneqq J$. We now prove the statement for general $i$. 
	
	Once again, our goal is to prove that $M^i$ and its complement are convex. According to Lemma \ref{DGH:lem:Zreal}, this is equivalent to realizability. Let $X$ be an $(i+1)$-set and $P_X$ the $i$-packet generated by $X$. Take $A, B, C \in P_X$ with $A <_{\lex} B <_{\lex} C$. There are two general cases to tackle:
	\begin{align}
		A, C \in M^i \tag{Case 1}\\
		B \in M^i \tag{Case 2}
	\end{align} 
	
	\textbf{Case 1}: Suppose that $A, C \in M^i$. We need to show that $B \in M^i$ as well. By definition, we are assuming that 
		\begin{equation}\label{Mireal:eqn1}
			A, C \in M^{i-1}_{s} \cup (M^{i-1} \setminus L^{i-1})_F 
		\end{equation}
		
	There are several subcases to consider depending upon precisely which component of $M^{i-1}$ $A$ and $C$ live in.
	\begin{align}
		A, C &\in M^{i-1}_s \tag{1.i} \\
		A \in M^{i-1}_s, C &\in (M^{i-1}\setminus L^{i-1})_F \text{ (or vice-a-versa)}\tag{1.ii} \\
		A, C &\in (M^{i-1}\setminus L^{i-1})_F\tag{1.iii}
	\end{align}
		
	\textbf{Case 1.i)} By definition, $L^i \coloneqq M^{i-1}_s$. Because $B$ is between $A$ and $C$ in the lexicographic order, the realizability of $L^i$ demands that $B \in L^i$ as well. Hence $B \in M^i$.
			
	\textbf{Case 1.ii)} As $M^i \subseteq M^{i-1}_s \cup M^{i-1}_F$, we are assuming that $A, C \in M^{i-1}_s \cup M^{i-1}_F$. Our inductive hypothesis implies that $M^{i-1}_s \cup M^{i-1}_F$ is a realizable $i$-set. Convexity requires that $B \in M^{i-1}_s \cup M^{i-1}_F$. By Lemma \ref{DGH:lem:ABpart}, one of the following is true:
			\begin{align}
				B &\in M_s^{i-1} \label{Mireal:C1:good1} \\
				B &\in M^{i-1}_F \cap L^{i-1}_F \label{Mireal:C1:bad1} \\
				B &\in M_F^{i-1} \cap L^{i-1}_s \label{Mireal:C1:bad2} \\
				B &\in M_{F}^{i-1} \cap L^{i-1}_{\emptyset} \label{Mireal:C1:good2} \\ 
				B &\in M_F^{i-1} \cap L^{i-1}_p \label{Mireal:C1:bad3}
			\end{align}

	If either Equation \eqref{Mireal:C1:good1} or Equation \eqref{Mireal:C1:good2} is true, then convexity holds. Fortunately, our previous work rules out Equations \eqref{Mireal:C1:bad1}, \eqref{Mireal:C1:bad2}, \eqref{Mireal:C1:bad3} from being true. Equation \eqref{Mireal:C1:bad1} is not possible because $L^{i-1}_F = \emptyset$ by Lemma \ref{DGH:lem:JSF}. Equation \eqref{Mireal:C1:bad2} is impossible because $M^{i-1}_F \cap L^{i-1}_s = \emptyset$ by Lemma \ref{DGH:lem:J1SJ2F}. If Equation \eqref{Mireal:C1:bad3} held, then $B \in L^{i-1}_p$ and $C \in L^{i-1}_{\emptyset}$. The realizability of $L^{i-1}_p$ would dictate that $A \in L^{i-1}_p$, which is impossible because $M^{i-1}_s \cap L^{i-1}_p = \emptyset$. Hence either Equation \eqref{Mireal:C1:good1} or \eqref{Mireal:C1:good2} must hold, placing $B \in M^i$.
			
	\textbf{Case 1.iii)} In this case, both $A, C \in M^{i-1}_F \cap L^{i-1}_{\emptyset}$. Lemma \ref{DGH:lem:parts} implies that $B \in M^{i-1}_F \cap L^{i-1}_{\emptyset}$ as well. Thus $B \in M^i$. 			
	
	\textbf{Case 2:} Now suppose that $B \in M^i$. We need to show that either $A \in M^i$ or $C \in M^i$ (or both). There are two subcases:
	\begin{align}
		B \in M^{i-1}_s \tag{2.i}\\
		B \in (M^{i-1}\setminus L^{i-1})_F \tag{2.ii}
	\end{align}
	\textbf{Case 2.i)} The realizability of $M^{i-1}_s$ demands that either $A \in M^{i-1}_s$ or $C \in M^{i-1}_s$. 
			
	\textbf{Case 2.ii)} The realizability of $M^{i-1}_s \cup M^{i-1}_F$ requires that either $A \in M^{i-1}_s \cup M^{i-1}_F$ or $C \in M^{i-1}_s \cup M^{i-1}_F$. We handle each case separately.
			
	Suppose that $A \in M^{i-1}_s \cup M^{i-1}_F$. Our goal is to show that either $A \in M^i$ or $C \in M^i$. Consider the statements analogous to Equations \eqref{Mireal:C1:good1}-\eqref{Mireal:C1:bad3} but for $A$ instead of $B$. Again, if either analogs to Equations \eqref{Mireal:C1:good1} or \eqref{Mireal:C1:good2} are true, then convexity holds as $A \in M^i$. We may rule out the analogs of Equations \eqref{Mireal:C1:bad1} and \eqref{Mireal:C1:bad2} as before. We are left to rule out the analog of equation \eqref{Mireal:C1:bad3}. 
	
	Suppose that $A \in M^{i-1}_F \cap L_{p}^{i-1}$. In order for this to imply that $M^i$ is non-convex, we would need $C \notin M^i$ as well. We claim that it is not simultaneously possible for $A \in M^{i-1}_F \cap L^{i-1}_p$ and $C \notin M^i$. 
			
	So far we have assumed that $P_X$ has the following (sub)segmentations:
		\begin{center}
			\begin{tabular}{ccccccccc}
				$\lbrace\ldots$ & $<$ & $A$ & $<$ & $B$ & $<$ & $C$ & $<$ & $\ldots\rbrace$ \\
				$\lbrace\ldots$ & $<$ & $M_{F}^{i-1}$ & $<$ & $M_F^{i-1}$ & $<$ & ?? & $<$ & $\ldots\rbrace$ \\
				$\lbrace\ldots$ & $<$ & $L^{i-1}_{p}$ & $<$ & $L^{i-1}_{\emptyset}$ & $<$ & ?? & $<$ & $\ldots\rbrace$
			\end{tabular}
		\end{center}
	By the assumed realizability of $M^{i-1}$, we know that
	\[C \in M^{i-1}_F \cup M^{i-1}_{\emptyset} \cup M^{i-1}_s \cup M^{i-1}_p\]
	Because $B \in M^{i-1}_F$, Lemma \ref{DGH:lem:Jreal} states that $C \notin M^{i-1}_{\emptyset}$. Thus if $C \notin M^i$, one of the following is true
		\begin{align}
				C &\in M^{i-1}_F \setminus (M^i \cap M_F^{i-1}), \label{Mireal:C2:bad1},\\
				C &\in M^{i-1}_p \label{Mireal:C2:bad2}.
		\end{align}	
		
	If Equation \eqref{Mireal:C2:bad1} holds, then applying Lemma \ref{DGH:lem:ABpart} with $A = L^{i-1}$ and $B = M^{i-1}$, along with Lemma \ref{DGH:lem:JSF} applied to $L^{i-1}$ and Lemma \ref{DGH:lem:J1SJ2F}, imply that $C \in M^{i-1}_F \cap L_p^{i-1}$. However, $A, C \in L^{i-1}_p$ and $B \in L^{i-1}_{\emptyset}$ violates the realizability of $L^{i-1}_p$.  
			
	If Equation \eqref{Mireal:C2:bad2} holds then Lemma \ref{DGH:lem:parts} implies that $C \in M^{i-1}_p \cap L^{i-1}_{\emptyset}$. Lemma \ref{DGH:lem:ABpart} along with the realizability of $M^{i-1}_p$ and $L^{i-1}_p$ implies that $P_X$ must have the following segmentation:
		\begin{equation*}
			\lbrace M_{F}^{i-1} \cap L^{i-1}_{p} < M^{i-1}_{F} \cap L^{i-1}_{\emptyset} < M^{i-1}_{p}\cap L^{i-1}_{\emptyset}\rbrace.
		\end{equation*}
	Such a segmentation is forbidden by Lemma \ref{DGH:lem:nobadLpMF}. As both equations \eqref{Mireal:C2:bad1} and \eqref{Mireal:C2:bad2} lead to contradictions, it is not possible to simultaneously have $A \in M^{i-1}_F \cup L^{i-1}_p$ and $C \notin M^i$. We conclude that either $A \in M^i$ or $C \in M^i$ whenever $A \in M^{i-1}_s \cup M^{i-1}_F$. So convexity holds. 
			
	We now assume that $C \in M^{i-1}_s \cup M^{i-1}_F$. Our goal is the same; to show that either $C \in M^{i}$ or $A \in M^i$. As before the only real cause for concern is the analog of equation \eqref{Mireal:C1:bad3}; that $C \in M^{i-1}_F \cap L^{i-1}_p$. If the analog of equation \eqref{Mireal:C1:bad3} holds, then we have the following (sub)segmentations:
		\begin{center}
		 	\begin{tabular}{ccccccccc}
		 		$\lbrace\ldots$ & $<$ & $A$ & $<$ & $B$ & $<$ & $C$ & $<$ & $\ldots\rbrace$ \\
		 		$\lbrace\ldots$ & $<$ & ?? & $<$ & $M_F^{i-1}$ & $<$ & $M_{F}^{i-1}$ & $<$ & $\ldots\rbrace$ \\
		 		$\lbrace\ldots$ & $<$ & ?? & $<$ & $L^{i-1}_{\emptyset}$ & $<$ & $L^{i-1}_{p}$ & $<$ & $\ldots\rbrace$
		 	\end{tabular}
		\end{center}
		
	Again, we wish to show that $A \in M^i$. Otherwise, one of the analogs of equations \eqref{Mireal:C2:bad1} or \eqref{Mireal:C2:bad2} must hold. As before, the realizability of $L^{i-1}_p$ prohibits the analog of equation \eqref{Mireal:C2:bad1} from being true. The analog of equation \eqref{Mireal:C2:bad2} also cannot be true. If it were, then the realizability of $M^{i-1}_p$ and $L^{i-1}_p$ would imply that $P_X$ have the following invalid segmentation (Corollary \ref{DGH:cor:badsets}):
		\begin{equation*}
				\lbrace M^{i-1}_p < M^{i-1}_{F}\rbrace.
		\end{equation*}
	It follows that neither the analog of equation \eqref{Mireal:C2:bad1} nor \eqref{Mireal:C2:bad2} are possible when $B \in (M^{i-1}\setminus L^{i-1})_F$ and $C \in M_s^{i-1}\cup M_F^{i-1}$. Hence, under those conditions, $C \in M^i$ and $M^i$ is convex. 
	
	This completes the proof that $M^i$ and its complement are convex. We have shown that whenever $A, C \in M^i$, then $B \in M^i$ as well. Furthermore, we have shown that whenever $B \in M^i$ either $A \in M^i$ or $C \in M^i$. 
\end{proof}

Now that we have shown that $M^i$ is realizable, we can proceed with proving that it is the inversion set of the unique maximal element in $\HB{i-1}{J}$. One aspect of this proof requires showing that $M^i$ is greater in the single step inclusion order than any other realizable $i$-set $K$ for which $L^i \subseteq K \subseteq M^i$. Similar to how we proved that $L^i$ was the inversion set of the unique minimal element, we will prove maximality by finding an admissible $M^i$ order that has $K$ has a prefix. Rather than explicitly writing down the order, it is easier to identify the order via its inversion set. Lemma \ref{DGH:lem:SSrel} tells us $K_s \cup M^i_s$ must be in the inversion set of any admissible $M^i$ order with $K$ as a prefix. If we can prove that $K_s\cup M_s^i$ is realizable, then we will be well on our way to proving the maximality of $M^i$.

\begin{lemma}\label{DGH:lem:KsMsreal}
	Let $i \geq 2$ and let $K$ be a realizable $i$-set for which $L^i \subseteq K \subseteq M^i$. Then $K_s \cup M^i_s$ is a realizable $(i+1)$-set.
\end{lemma}
\begin{proof}
	Instead of proving realizability we will prove convexity. Let $P_X$ be an $i$-packet containing $A \leq_{\lex} B \leq_{\lex} C$. We need to prove that neither of the following occurs
	\begin{align}
		A, C \in K_s \cup M^i_s &\text{ but } B \notin  K_s \cup M^i_s \label{KsMsreal:c1} \tag{*}\\
		A, C \notin K_s \cup M^i_s &\text{ but } B \in K_s \cup M^i_s \label{KsMsreal:c2}\tag{**}
	\end{align}
	
	First we show that \eqref{KsMsreal:c1} is impossible. Assume that $A, C \in K_s \cup M^i_s$. By Lemma \ref{DGH:lem:ABpart}, we know that 
	\begin{equation}\label{KsMsreal:KsMs}
		K_s \cup M_s^i = (K_s \cap M^i_s) \cup (K_s \cap M^i_F) \cup (K_{\emptyset} \cup M^i_s)
	\end{equation}
	Equation \eqref{KsMsreal:KsMs} implies that there are nine distinct ways for our assumption to hold. Each possibility is listed in Table \ref{MsKsreal:C1:Segs}.
	\begin{center}
		\begin{table}[h]
		\begin{tabular}{c|ccc|ccc|ccc}
			& A & B & C & A & B & C & A & B & C \\
		  \hline K & $K_s$ & ?  & $K_s$ & $K_s$ & ? & $K_s$ & $K_s$ & ? & $K_{\emptyset}$ \\
		  M & $M^i_s$ & ? & $M^i_s$ & $M^i_s$ & ? & $M^i_F$ & $M^i_s$ & ? & $M^i_s$ \\
		  \hline K  & $K_s$ & ?  & $K_s$ & $K_s$ & ? & $K_s$ & $K_s$ & ? & $K_{\emptyset}$ \\
		  M & $M^i_F$ & ? & $M^i_s$ & $M^i_F$ & ? & $M^i_F$ & $M^i_F$ & ? & $M^i_s$ \\
		  \hline K & $K_{\emptyset}$ & ? & $K_s$ & $K_{\emptyset}$ & ? & $K_s$ & $K_{\emptyset}$ & ? & $K_{\emptyset}$ \\
		  M & $M^i_s$ & ? & $M^i_s$ & $M^i_s$ & ? & $M^i_F$ & $M^i_s$ & ? & $M^i_s$ 
		\end{tabular}
		\caption{\label{MsKsreal:C1:Segs} Possible ways to have $A, C \in K_s \cup M^i_s$.}
		\end{table}
	\end{center}
	Beginning with an index of 0, the $(0,0), (0,1), (0,2), (1,0), (1,1), (2,0), \text{and } (2,2)$ entries of Table \ref{MsKsreal:C1:Segs} are all require $B \in K_s \cup M^i_s$ in order to maintain the realizability of $K_s$ or $M^i_s$. 
	
	It is impossible for entry $(1,2)$ to occur. If it did then Lemma $\ref{DGH:lem:parts}$ implies that $P_X$ is segmented as 
	\[P_X = \lbrace M^i_F < M^i_s\rbrace\]
	However, Corollary \ref{DGH:cor:badsets} forbids such a packet segmentation.
	
	Entry $(2,1)$ is the only one which we have not addressed. Lemma \ref{DGH:lem:parts} implies that 
	\begin{align*}
		B &\in K_{\emptyset} \cup K_s \\
		B &\in M^i_s \cup M^i_F
	\end{align*}
	We only run into issues if $B \in K_{\emptyset} \cap M^i_F$. But if this is true, then $P_X$ is segmented as:
	\[P_X = \lbrace K_{\emptyset} \cap M^i_s < K_{\emptyset} \cap M^i_F < K_s \cap M^i_F\rbrace \]
	Such a segmentation is forbidden by Lemma \ref{DGH:lem:nobadK0MF}. So whenever entry $(2,1)$ is true, $B \in K_s \cup M^i_s$. Since all valid entries in Table \ref{MsKsreal:C1:Segs} imply that $B \in K_s \cup M^i_s$, it follows that \eqref{KsMsreal:c1} is impossible.
	
	We now rule out \eqref{KsMsreal:c2}. Assume that $B \in K_s \cup M^i_s$. If $B \in K_s$, then the realizability of $K_s$ requires that either $A \in K_s$ or $C \in K_s$. Likewise if $B \in M^i_s$, either $A \in M^i_s$ or $C \in M^i_s$. Therefore one or both of $A \in K_s \cup M^i_s$ and $C \in K_s \cup M^i_s$ is true. Hence \eqref{KsMsreal:c2} is impossible. As both \eqref{KsMsreal:c1} and \eqref{KsMsreal:c2} are impossible, $K_s \cup M_s$ is realizable. 
\end{proof}

\begin{lemma}\label{DGH:lem:unimax}
	The element $r \in B_{i-1}(J)$ with inversion set $M^i$ is the unique maximal element of $\HB{i-1}{J}$.
\end{lemma}
\begin{proof}
	To complete this proof, we must show that there exists $r \in \HB{i-1}{J}$ with $\Inv(r) = M^i$ and that such an $r$ is the unique maximal element in $\HB{i-1}{J}$. 
	
	We begin with existence. To prove existence we will find an admissible $M^{i-1}$ order $\rho$ for which $L^{i-1}$ is a prefix and for which $\Inv(\rho) = M^i$. Then $r \coloneqq \pi(\rho)$.
	
	Lemma \ref{DGH:lem:Mireal} asserts that $M^i$ is a realizable $i$-set. By definition, $M^{i-1}_s \subseteq M^i \subseteq M^{i-1}_s \cup M^{i-1}_F$. Therefore Theorem \ref{DGH:thm:PathZEquiv} asserts that there is a $\tilde{r} \in \epaths{\emptyset}{M^{i-1}}{i-2}$ for which $\Inv(r) = M^i$. Pick $\rho \in \tilde{r}$. By definition, $\rho \in A_{M^{i-1}}(n,k)$. 
	
	We claim that, up to a sequence of elementary equivalences, $L^{i-1}$ is a prefix of $\rho$. As we have seen before, this is because $L^{i-1}_s \subseteq M^i$ and $L^{i-1}_p \cap M^i = \emptyset$ (Lemma \ref{DGH:lem:Zequiv}). Without loss of generality, assume that $\rho$ has $L^{i-1}$ as a prefix. Then $r \coloneqq \pi(\rho) \in \HB{i-1}{J}$ and $\Inv(r) = M^i$ as desired.
	
	We now show that $r$ is the unique maximal element in $\HB{i-1}{J}$. Take any other $r' \in \HB{i-1}{J}$. We need to show that $r' \leq_{\MS} r$. Theorem \ref{DGH:thm:PathZEquiv} instead allows us to prove that $\Inv(r') \leq_{\SSO} \Inv(r)$, or that $K \leq_{\SSO} M^i$ for any realizable $i$-set $L^i \subseteq K \subseteq M^i$. This, of course, is equivalent to finding an admissible $M^i$ order for which $K$ is a prefix. As we saw in the proof of Theorem \ref{DGH:thm:PathSnk} this is equivalent to proving that $K_s \cup M^{i}_s$ is a realizable $(i+1)$-set, which is precisely the result of Lemma \ref{DGH:lem:KsMsreal}.
\end{proof}

This completes the proof of Theorem \ref{DGH:thm:DGHBruhat}. Lemmas \ref{DGH:lem:unimin} and \ref{DGH:lem:unimax} establish that $\HB{i-1}{J}$ have unique minimal and maximal elements whose inversion sets, respectively, are $L^i$ and $M^i$. 

\section{Issues with a general $k$}

Theorem \ref{DGH:thm:DGHBruhat} was only proved a realizable $2$-set $J$. In particular, the proofs of Lemmas \ref{DGH:lem:nobadLpMF} and \ref{DGH:lem:nobadK0MF} relied on $J$ being a realizable $2$-set.
It is unknown if Theorem \ref{DGH:thm:DGHBruhat} holds for a general realizable $k$-set. The proof certainly cannot. In particular, Lemma \ref{DGH:lem:nobadK0MF} fails for some certain realizable $k$-sets $J$. The following counterexample is due to Ben Elias.

\begin{example}
	The counterexample occurs when $n = 9$ and $k = 5$. Let $J = M^5 = C(9,5)\sdiff{12345}$. Then
	\begin{align}
		L^6 &= \lbrace X \in C(9,6) \mid \{12345\} \subset X\rbrace\\
		\begin{split}
		M^6 &= L^6 \cup \lbrace X \in C(9,6) \mid \{12345\}\nsubseteq X \rbrace \\
		&= C(9,6).
		\end{split}
	\end{align}
	Recall that $L^6$ and $M^6$ are, respectively, the inversion sets of the unique minimal and maximal elements of the second Bruhat order of $J$. Continuing, we get that
	\begin{align}
		L^7 &= \emptyset \\
		\begin{split}
			M^7 &= L^7 \cup (M^6\setminus L^6)_F \\
				&= \{X \in C(9,7) \mid P_X \subseteq M^6\setminus L^6\} \\
				&= \{X \in C(9,7) \mid L^6 \cap P_X = \emptyset\} \\
				&= \{X \in C(9,7) \mid \{12345\}\nsubseteq X\}
		\end{split}
	\end{align}
	It is at this point where our proof strategy breaks down. Consider the realizable $7$-set 
	\[K = \{2356789, 2456789, 3456789\}.\]
	Clearly $L^7 \subset K \subset M^7$. However, $K_s \cup M_s$ is not a realizable $8$-set. In particular, the $8$-packet generated by $\{123456789\}$ is segmented as in \eqref{nobadK0MF:eqn:seg}. To quickly see this note that 
	\begin{align*}
		\{12345789\} <_{\lex} \{1234&6789\} <_{\lex} \{12356789\}\\
		\{12345789\} &\in K_{\emptyset} \cap M^7_s \\
		\{12346789\} &\in K_{\emptyset} \cap M^7_F \\
		\{12356789\} &\in K_s \cap M^7_F.
	\end{align*}
	This does not serve as a counterexample to our general result. There is no realizable $2$-set for which $L^7 = \emptyset$ and $M^7 = \{X \in C(9,7) \mid \{12345\}\nsubseteq X\}$.
\end{example} 
	
This also does not mean that it is impossible to generalize Theorem \ref{DGH:thm:DGHBruhat} to higher $k$'s; only that our approach of proving that $K_s \cup M^i_s$ is realizable fails. In fact, Theorem \ref{DGH:thm:PathSnk} proves that the second higher Bruhat order of a realizable $k$-set does have a unique minimal and maximal element. This holds for any $k$. In order to further extend higher Bruhat orders to arbitrary realizable $k$-sets, it is likely that a more complicated formula than $K_s \cup M^i_s$ will be needed, or perhaps even an entirely different proof strategy.	

\section{Conclusion and Next Steps}

\subsection{Summary of Results}

Theorem \ref{DGH:thm:DGHBruhat} is the main result of this paper. It extends Manin and Schechtman's original result on higher Bruhat orders to higher Bruhat orders defined for non-longest words in $S_n$. This is accomplished using the isomorphism between realizable $2$-sets $J$ and words $w$ in $S_n$ to define the second Bruhat order of a non-longest word as equivalence classes of paths from $\emptyset$ to $J$ in the weak Bruhat order of $S_n$. Subsequent higher Bruhat orders are then defined as equivalence classes of subposets of paths in $B(n,k)$. By defining higher Bruhat orders in this manner, we are able prove an analog of Ziegler isomorphism between realizable $(k+1)$-sets and elements of $B(n,k)$. This allows us to take full advantage of Manin and Schechtman as well as Ziegler's original results in order to prove that non-longest word higher Bruhat orders have unique minimal and maximal elements, as well as satisfy a suitable generalization of the Manin-Schechtman correspondence. 

As discussed, our proofs crucially rely on beginning with a realizable $2$-set. In particular, we prove that $\HB{i-1}{J}$ has a unique maximal element by proving that $K_s \cup M^i_s$ is realizable for all realizable $i$-sets $L^i \subseteq K \subseteq M^i$. This allowed us to prove that $K \leq_{SS} M^i$ by finding an admissible $M^i$ order with $K$ as a prefix. The proof that $K_s \cup M^i_s$ is realizable relied on a contradiction that we only showed occurred when $J$ is a realizable $2$-set. In general, it is not clear if an analogous contradiction occurs when $J$ is not a realizable $2$-set. Instead, we are only able to prove that the second higher Bruhat order of an arbitrary realizable $k$-set has a source and sink.

\subsection{Affine Case}\label{DGH:sec:Affine}

Fully extending these results (or providing a counterexample) is not the only direction for future research. Shelley-Abrahamson and Vijaykumar's results on second and third Bruhat orders to Type B Weyl groups suggest that Manin and Schechtman's results may be extend to other Weyl or Coxeter groups. Affine type A Coxeter groups are a good candidate. 

\begin{definition}\label{DGH:def:CZk}
	For any $k > 1$, let $C(\mathbb{Z},k)$ denote the size $k$ subsets of $\mathbb{Z}$. We identify $C(\mathbb{Z},k)$ with the image of $\mathbb{Z}^{k}$ under the natural action of $S_k$. 
\end{definition}

Fix $N$ and recall the Coxeter presentation of $W^{aff}_{N-1}$.

\begin{equation}\label{DGH:def:affine}
	\begin{split}
	W^{aff}_{N-1} = \langle s_1\ldots s_N \mid &s_is_js_i = s_js_is_j \text{ if } i \equiv j\pm 1\mod N, \\
	&s_is_j = s_js_i \text{ otherwise}, s_i^2 = 1 \rangle
	\end{split}
\end{equation} 

We can embed $W^{aff}_{N-1}$ into the set of permutations of $\mathbb{Z}$ via the following map:
\begin{align*}
	s_{1} &\mapsto \ldots (1,2)(1+N,2+N)\ldots \\
	s_{2} &\mapsto \ldots (2,3)(2+N,3+N)\ldots \\
	&\vdots \\
	s_N &\mapsto \ldots (0,1)(N,N+1)\ldots
\end{align*}

In particular, the image of $W^{aff}_{N-1}$ is a subset of the $N$-periodic permutations of $\mathbb{Z}$. 

\begin{definition}
	Let $w$ be a permutation of $S_n$. If $w(i+N) = w(i)+N$ for all $i \in \mathbb{Z}$, then $w$ is an $N$-periodic permutation of $\mathbb{Z}$. The set of all $N$-periodic permutations of $\mathbb{Z}$ is denoted $\Sigma_N$.
\end{definition}

There is an $N$-periodic permutation $w_{t}$ defined by $w_t(i) = i+1$ which generates the subgroups of translations of $\mathbb{Z}$. This subgroup is isomorphic to $\mathbb{Z}$. There is an isomorphism $\Sigma_N \cong \mathbb{Z} \rtimes W_N^{aff}$. 

The action of $W^{aff}_{N}$ on $\mathbb{Z}$ allows us to define the following set for any $\tilde{w} \in W^{aff}_{N-1}$.
\begin{equation}\label{DGH:def:Inv'}
	\Inv'(\tilde{w}) \coloneqq \lbrace (x,y) \in \mathbb{Z}\times\mathbb{Z} \mid x < y, \tilde{w}(y) < \tilde{w}(x)\rbrace
\end{equation}

Periodicity necessitates that $(x, y) \in \Inv'(\tilde{w})$ whenever $(x+N, y+N) \in \Inv'(\tilde{w})$. Let $N\mathbb{Z}$ be the additive subgroup of $\mathbb{Z}$ generated by $N$. Define the diagonal action of $N\mathbb{Z}$ on $\mathbb{Z}\times\mathbb{Z}$ by
\begin{equation}\label{DGH:eqn:diagact}
	(Nk).(x,y) = (x+Nk, y+Nk) \forall k \in \mathbb{Z}
\end{equation}
As we just observed, $\Inv'(\tilde{w})$ is invariant under the diagonal action of $N\mathbb{Z}$. 

\begin{notation}
	Let $\sim_d$ be the equivalence relation generated by the diagonal action of $N\mathbb{Z}$ in equation \eqref{DGH:eqn:diagact}. Then $\tilde{C}(N,k) \coloneqq C(\mathbb{Z},k)\big/ \sim_d$. An element of $\tilde{C}(N,k)$ is written as $[x_1\ldots x_k]$.
\end{notation}

\begin{definition}\label{DGH:def:AffInv}
	 If $\pi_d: \mathbb{Z}^2 \rightarrow \mathbb{Z}^2\big/ \sim_d$ is the projection map, then we define
	\[
		\Inv(\tilde{w}) = \pi(\Inv'(\tilde{w}))
	\]
	Moreover, under the identification in Definition \ref{DGH:def:CZk}, $\Inv(\tilde{w})$ is a subset of $\tilde{C}(N,2)$.
\end{definition}

As defined, $\Inv(\tilde{w})$ never contains pairs $[x,y]$ where $x \equiv y \mod N$. This is because whenever $x \equiv y \mod N$, then periodicity implies that $\tilde{w}(y) = \tilde{w}(x) + (y-x) > \tilde{w}(x)$. So $\Inv(\tilde{w})$ avoids the $N\mathbb{Z}$-invariant subset $Y_2 \coloneqq \lbrace [x,y] \mid x\equiv y \mod N \rbrace$. While we could work with $\tilde{C}(N,2)\setminus Y_2$, we will see momentarily that it is convenient to at least consider $\tilde{C}(N,2)$ in its entirety.

A new difficulty that arises in the affine setting is determining the lexicographic order on $\tilde{C}(N,2)$. For example, let $N = 3$ and consider $[13]$ and $[34] = [01]$. Consider the three strands with labels $\{1,3,4\}$ and their packet $P_{[134]} = \{[13] < [14] < [34]\}$ ordered lexicographically by their representatives. Within this packet $[13] < [34]$. Now consider the packet $P_{[013]} = \{[01] < [03] < [13]\}$. Within this packet $[34] = [01] < [13]$. This suggests that it is foolish to consider total orders on $\tilde{C}(N,2)$ as there is not total order on $\tilde{C}(N,2)$ that respects the lexicographic order on each $2$-packet. Instead, we must limit ourselves to total orders on finite realizable subsets of $\tilde{C}(N,2)$.

Again consider $P_{[134]}$. The middle element $[14]$ is in $Y_2$. Therefore it never appears in an inversion set. Using precisely the same argument as in the non-affine setting (see the paragraph below \ref{DGH:ex:stutst}), $\Inv(\tilde{w})$ must intersect $P_{[134]}$ as either a prefix or a suffix. The same is true of $P_{[013]}$. Hence no inversion set contains both $[13]$ and $[34]$, and we need not concern ourselves with their ordering. This is true in general if we exclude $Y_2$ from any realizable $2$-set.

If we are trying to pin down a characterization of the inversion sets of periodic permutations (i.e. realizability), then it is clear that $P_{[134]}$ plays a role. However, $P_{[134]}$ will never be full, so it can never be flipped in any higher Bruhat order. This is analogous to how reduced expressions for non-longest words in $S_n$ have associated $J_p$ and $J_s$ sets.

Having illustrated some of the subtleties, we now give the main definitions. 

\begin{definition}\label{DGH:def:diagG}
	For $k \geq 1$, define the \emph{diagonal subgroup of $\mathbb{Z}^k$} or $\Delta_kN\mathbb{Z}$ as the subgroup generated by $(N,\ldots,N)$. Then define $\tilde{C}(N,k) \coloneqq C(\mathbb{Z},k)\big/\Delta_kN\mathbb{Z}$. 
\end{definition}

\begin{definition}\label{DGH:def:Yk}
	Let $Y_k \coloneqq \{[x_1,\ldots,x_k] \mid x_i\equiv x_j \text{ for some } i,j\}$. Then $Y_k \subset \tilde{C}(N,k)$. Observe that $Y_k = \tilde{C}(N,k)$ for $k \geq N$.
\end{definition}

\begin{definition}\label{DGH:def:affPack}
	Given an element $X = [x_1 < \ldots < x_{k+1}]$ in $\tilde{C}(N,k+1)$, the \emph{k-packet generated by X} is defined as 
	\begin{equation*}
		P_X \coloneqq \{[\widehat{x}_i] \mid 1 \leq i \leq k+1\}
	\end{equation*}
	The $k$-packet $P_X$ is ordered lexicographically by $[\widehat{x}_i] < [\widehat{x}_j]$ whenever $i > j$. 
\end{definition}
The definition of the lexicographic order of $P_X$ in Definition \ref{DGH:def:affPack} is independent of the chosen representative. It is invariant under translations. 

\begin{definition}\label{DGH:def:affReal}
	A finite subset $J \subset \tilde{C}(N,k)\setminus Y_k$ is realizable if its intersection with every $k$-packet is either a prefix or suffix of the $k$-packet. We define the following $(k+1)$-sets for any realizable $k$-set:
	\begin{enumerate}
		\item $J_p \coloneqq \{X \in \tilde{C}(N,k+1) \mid P_X \cap J \text{ is a proper prefix of }P_X\}$
		\item $J_s \coloneqq \{X \in \tilde{C}(N,k+1) \mid P_X \cap J \text{ is a proper suffix of }P_X\}$
		\item $J_F \coloneqq \{X \in \tilde{C}(N,k+1) \mid P_X \cap J \text{ is full in }P_X\}$
		\item $J_{\emptyset} \coloneqq \{X \in \tilde{C}(N,k+1) \mid P_X \cap J \text{ is empty in }P_X\}$
	\end{enumerate}
\end{definition}

Note that $\tilde{C}(N,k)\setminus Y_k$ is \emph{not} a realizable $k$-set, due to the existence of packets like $P_{[134]}$. 

\begin{definition}\label{DGH:def:affadm}
	A total order $\rho$ on a realizable $k$-set $J$ is \emph{admissible} if it induces
	\begin{enumerate}
		\item ... the antilexicographic order on $P_X$ when $X \in J_s$,
		\item ... the lexicographic order on $P_X$ when $X \in J_p$,
		\item ... either the lexicographic or antilexicographic order on $P_X$ when $X \in J_F$.
	\end{enumerate}
\end{definition}

Packet flips and elementary equivalences still make sense and are defined as they were in the non-affine case. 

Numerous examples suggest that, up to elementary equivalence, the collection of admissible orders on $\Inv(\tilde{w})$ have a well-defined source and sink, as do paths between paths etc. In particular, the unique minimal $J$-order appears to have $J_s$ as its inversion set, while the unique maximal $J$-order has $J_s \cup J_F$ as its inversion set. 

There are some obstructions to directly copying the finite case proofs over to the affine case. Unfortunately, Ziegler's correspondence between the single step inclusion order and Manin-Schechtman order is no longer available in the affine case. In other words, if two inversion sets differ by a single element $X$, it is no longer clear that you can flip $P_X$ (see Remark \ref{DGH:rmk:Ziegthm}). Even if Ziegler's bijection could be established, the notion of ``extending orders to maximal chains'', so central to our proof, is no longer applicable in the affine case. Another approach is to replicate Manin and Schechtman's argument using ``good orders.'' But again, the author is unsure as to how to adapt this to the affine case. A third approach is based upon using the results of this paper.

Unlike $S_n$, $\tilde{A}_{N}$ has no longest word. A path in the weak Bruhat graph of $\tilde{A}_{N}$ from the identity element $e$ to $\tilde{w}$ is therefore somewhat similar to a path in the weak Bruhat graph of $S_n$ from $e$ to $w$. If it were somehow possible to fully encode $\Inv(\tilde{w})$ as the inversion set of some $w \in S_{n'}$ for $n' > N$, then it might be possible to use Theorem \ref{DGH:thm:DGHBruhat} to prove that the second Bruhat order of $\tilde{w}$ has a unique minimal and maximal element. Even higher Bruhat orders could then be defined based on the higher Bruhat orders of $w$. We leave this for future research efforts!

\appendix

\section{Generating Higher Bruhat Orders}\label{DGH:Sec:App}

This appendix is devoted to providing and explaining a simple algorithm for generating higher Bruhat orders. At its core is a greedy algorithm for determining which packets are flippable in a given admissible $k$-order. Once you can determine which packets are flippable in a given admissible $k$-order you can generate the entire higher Bruhat order by running the algorithm on $\rho_{\lex}$, as well as the generated admissible $k$-orders.

The algorithm consists of a main routine and three subroutines. Pseudocode for the main routine is found in Algorithm \ref{DGH:alg:Flip}. This portion of the algorithm checks every $k$-packet $P$ to see if $P$ is in the lexicographic order and if it can be brought together by a sequence of elementary equivalences. If $P$ can be brought together, then we store the equivalent admissible $k$-order in which $P$ forms a continuous block. This new admissible $k$-order is then used to check if any future $k$-packets can be brought together. 

\begin{algorithm}
	\caption{Finding Flippable Packets}\label{DGH:alg:Flip}
	\begin{algorithmic}[1]
		\Procedure{findFlip}{admOrder, packets} \Comment{}
		\State \textbf{array} newOrders
		\For{p in packets}
		\State	indices = [index of each entry in p]
		\State slice = admOrder[indices[0]:indices[last]] \Comment{minimal chain containing p}
			\If{sorted(indices) equals indices} \Comment{check that p is in lex}
			\State	together, subexp = comeTogether(slice, p)
				\If{together is True}
				\State	flip(subexp, p) \Comment{reverse the order on p}
				\State	newOrder = admOrder
				\State	newOrder[indices[0]:indices[last]] = subexp 
				\State	newOrders.append(neworder)
				\EndIf
			\EndIf
		\EndFor	
		\EndProcedure
	\end{algorithmic}
\end{algorithm}

Algorithm \ref{DGH:alg:Flip} calls Algorithm \ref{DGH:alg:Come} every time it checks if a $k$-packet $P$ can be brought together. The parameters of Algorithm \ref{DGH:alg:Come} are $P$ as well as the minimal chain $\overline{P}$ containing $P$ in your admissible order. It is assumed that some prefix of $P_X$ is also a prefix of $\overline{P}$. Algorithm \ref{DGH:alg:Come} works by checking if the elements $B \coloneqq \{x \in \overline{P} \mid x \notin P\}$ can be commuted to the beginning or end of $\overline{P}$. By initially iterating over $B$ lexicographically, this reduces to checking if each entry $x \in B$ can be commuted down past $L_x \coloneqq \{y \in P \mid y < x\}$ or, if that is not possible, up past all the remaining elements of $\overline{P}$. In order to simplify matters, every time $x$ is commuted down past $L_x$ we update our order so that $x < y$ for all $y \in L_x$. 

\begin{algorithm}
	\caption{Check if an element commutes up past a prefix}\label{DGH:alg:moveDown}
	\begin{algorithmic}[1]
		\Procedure{moveDown}{prefix,elt}
		\For{x in prefix}
		\If{x and elt share a packet}
		\State \textbf{return False}
		\EndIf
		\EndFor
		\State \textbf{return True}
		\EndProcedure
	\end{algorithmic}
\end{algorithm}

Because we update our order as we go, it is very easy to check if $x$ commutes past $L_x$. You only need to check that $x$ is not in any shared packets with the entries of $L_x$. This is accomplished by Algorithm \ref{DGH:alg:moveDown}. It is harder to check if $x$ commutes up. We cannot just check if $x$ commutes past $G_x \coloneqq \{y \in P \mid x < y\}$. This is because $x$ may be blocked by (i.e. may not commute with) an element $x' \in B$ or an element of $G_x$. Of course, if $x$ is blocked by an element in $G_x$, then $P$ cannot be brought together. However, if $x$ is blocked by $x' \in B$, then we may be able to commute $x'$ up before commuting $x$ up. This check is accomplished by Algorithm \ref{DGH:alg:bubble}.

\begin{algorithm}
	\caption{Checking if a packet can be brought together}\label{DGH:alg:Come}
	\begin{algorithmic}[1]
		\Procedure{comeTogether}{exp, packet}
			\State pointer = 0 
			\State curPos = 0
			\State endPos = $\text{len(exp)} - 1$
		
			\While{True}
				\If{$\text{endPos} - (\text{curPos}-\text{pointer}+1) == \text{len(packet)}$}
					\State	\textbf{return} \textbf{True}, exp, endPos \Comment{packet brought together}
				\EndIf
				\State nextPos = curPos + 1
				\State nextElt = exp[nextPos]
				\If{nextElt in packet}
					\State pointer += 1
					\State curPos +=1 
					\State \textbf{continue}
				
					\Else
						\If{\textbf{moveDown}(packet[0:pointer+1],nextElt)}\State \Comment{Blocking elt moved down}
							\State \textbf{move} nextElt up past exp[curPos-pointer:curPos]
							\State curPos += 1
						\Else
							\State slice = exp[nextPos, endPos+1]
							\State bubbled, subexp = \textbf{bubbleUp}(nextElt, slice, packet)
							\If{bubbled is \textbf{True}} \Comment{Blocking elt bubbled up}
								\State exp[slice] = subexp
								\State endPos = exp.index(packet[-1])
							\Else	
								\State \textbf{return False}, exp, endPos
							\EndIf
						\EndIf
				\EndIf
			\EndWhile
		\EndProcedure
	\end{algorithmic}
\end{algorithm}

Algorithm \ref{DGH:alg:bubble} takes as inputs $x$, a suffix of $\overline{P}$(as subexp), and $P$. If Algorithm \ref{DGH:alg:bubble} is successful, then it returns a modification of subexp with all the required commutations to put $x$ below $G_x$. Algorithm \ref{DGH:alg:bubble} begins by commuting $x$ as far up as is possible. This proceeds so long as $x$ is not in a shared packet with the neighbor below it. If $x$ is blocked by a neighbor in a shared packet, then Algorithm \ref{DGH:alg:bubble} calls itself to try to bubble up the blocking neighbor. If this is successful, then we continue trying to bubble up $x$. Otherwise, we return false indicating the failure to bubble up $x$. If $x$ commutes up past the rest of packet, then Algorithm \ref{DGH:alg:bubble} returns true along with a modified subexpression. This subexpression is a modification of the original input string which contains all of the commutatation moves performed by the algorithm. We expect some elements of $B$ to be commuted down past $P_X$, while the initial $x$ and, possibly, other members of $B$ are commuted up past $P_X$. These elements can now be ignored as they need not commute with future elements in order to bring $P_X$ together. Algorithm \ref{DGH:alg:Come} keeps track of this information, the index of the final entry of $P_X$ within the expression, with the pointer endPos. Likewise, the index of the first entry of $P_X$ within the expression is given by $\text{curPos} - \text{pointer}$.

\begin{algorithm}
	\caption{Try to bubble up a blocking element}\label{DGH:alg:bubble}
	\begin{algorithmic}[1]
		\Procedure{bubbleUp}{headElt, subexp, packet}
		\State curPos = 0
		\State endPos = len(subexp) - 1
		\State endElt = packet[-1]
		\While{$curPos < endPos$}
			\State nextPos = curPos + 1
			\State nextElt subexp[nextPos]
			
			\If{headElt and nextElt don't share a packet}
				\State subexp = flip(headElt, nextElt)
				\State curPos += 1
				\If{nextElt is endElt}
					\State endPos -= 1
				\EndIf
			\Else
				\If{nextElt in packet}
					\State \textbf{return False}, subexp
				\Else	
					\State bubbled, subsubexp = \textbf{bubbleUp}()
					\State slice = [nextPos:endPos+1]
					\State subexp[slice] = subsubexp
					\If{bubbled is \textbf{True}}
						\State endPos -= 1
					\Else
						\State \textbf{return False}, subexp
					\EndIf
				\EndIf
			\EndIf	
		\EndWhile
		\State \textbf{return True}, subexp
		\EndProcedure
	\end{algorithmic}
\end{algorithm}

\begin{lemma}\label{DGH:lem:BubbleWorks}
	Given a total order $\rho$ on a collection of $k$-sets $X$, Algorithm \ref{DGH:alg:bubble} will successfully identify if there exists a sequence of elementary equivalences moving $x_{\min} \coloneqq \min(\rho)$ above $x_{\max} \coloneqq \max(\rho)$.
\end{lemma}
\begin{proof}
	We induct on the size of $X$. When $\vert X\vert = 1$, there is nothing to do. More generally, it is evident that algorithm \ref{DGH:alg:bubble} works if $x_{\min}$ does not share a packet with any other element of $X$. So we may assume that there exists at least one $y \in X$ such that $x_{\min}$ and $y$ are in a common packet $P_Y$. Without loss of generality, we may assume that $x_{\min}$ and $y$ are adjacent under $\rho$. 
	
	Because $x_{\min}$ and $y$ are in a shared packet, it is not possible to move $y$ below $x_{\min}$. The only way to move $x_{\min}$ above $x_{\max}$ is by moving $y$ above $x_{\max}$ as well. By our induction hypothesis, Algorithm \ref{DGH:alg:bubble} will correctly accomplish this for $X\sdiff{x_{\min}}$. If we are unable to move $y$ up past $x_{\max}$, then we return false. If we are able to move $y$ up past $x_{\max}$, then we are left with running Algorithm 4 on $X\sdiff{y}$. As $\vert X\sdiff{y}\vert < \vert X\vert$ our induction hypothesis guarantees that our algorithm will work. 
\end{proof}

\begin{theorem}
	Given an admissible $k$ order $\rho$, algorithms (\ref{DGH:alg:Flip})-(\ref{DGH:alg:bubble}) will return 
	\begin{equation}\label{DGH:eqn:Nlex}
	N([\rho])_{\lex} \coloneqq \{X \in N([\rho]) \mid \rho\vert_{P_X} = \rho_{\lex}\vert_{P_X}\}.
	\end{equation}
\end{theorem}
\begin{proof}
	Algorithm \ref{DGH:alg:Flip} is essentially a wrapper function iterating over each packet $P_X$. The ``meat'' of the proof is that Algorithms \ref{DGH:alg:Come} and \ref{DGH:alg:bubble} will correctly identify when $X \in N([\rho])_{\lex}$. 
	
	Let $\overline{P_X}$ be the minimal chain in $\rho$ which contains $P_X$. We will induct on $\vert \overline{P_X}\setminus P_X\vert$. If $\vert \overline{P_X}\setminus P_X\vert = 0$, then $P_X$ already forms a chain. Algorithm \ref{DGH:alg:Come} identifies this by incrementing ``pointer'' and ``'curPos'' until ``pointer'' reaches the end of $P_X$ and signals to return True. This proves the base case.
	
	By definition, $x_{\min} \coloneqq \min(\overline{P_X})$ and $x_{\max} \coloneqq \max(\overline{P_X})$ are both in $P_X$. Note that $X \in N([\rho])$ is equivalent to being able to, through a sequence of elementary equivalences, move every $x \in \overline{P_X}\setminus P_X$ down past $x_{\min}$ or up past $x_{\max}$. 
	
	A greedy way to accomplish this is to seek out the minimal element of $\overline{P_X}$ which is not in $P_X$. Call this element $y$. Algorithm \ref{DGH:alg:Come} first attempts to move $y$ down. If this is successful, then Algorithm \ref{DGH:alg:Come} has produced a new total order on $\overline{P_X}$ in which the minimal chain containing $P_X$ has one fewer element than before. The rest of Algorithm \ref{DGH:alg:Come} runs on this new minimal chain. Hence it is gauranteed to work by our inductive hypothesis.
	
	If $y$ cannot move down, then the only way to clear out $y$ is by bubbling it up. Algorithm \ref{DGH:alg:Come} calls Algorithm \ref{DGH:alg:bubble}, whose proof of behavior is contained in Lemma \ref{DGH:alg:bubble}. If Algorithm \ref{DGH:alg:bubble} returns True, then Algorithm \ref{DGH:alg:Come} proceeds, again focusing on a smaller minimal chain. Therefore our inductive hypothesis guarantees its correctness. If Algorithm \ref{DGH:alg:bubble} returns False, then Algorithm \ref{DGH:alg:Come} also returns False, indicating that $P_X$ cannot be brought together.
\end{proof}

\bibliography{BruhatBib}
\bibliographystyle{plain}

\end{document}